\documentclass[11pt]{article}
\usepackage{fullpage}
\usepackage[toc,page]{appendix}
\usepackage{amssymb}
\usepackage{amsmath}
\usepackage{amsthm}
\usepackage{multirow}
\usepackage{complexity}
\usepackage{enumerate}
\usepackage{graphicx}
\usepackage{mathrsfs}
\usepackage[utf8]{inputenc}
\usepackage[backend=biber, style=alphabetic, sortcites=true, maxalphanames=4,maxbibnames=99,doi=false,isbn=false,url=false,eprint=false]{biblatex}
\usepackage{thm-restate}

\usepackage{color}
\usepackage[pdfstartview=FitH,pdfpagemode=UseNone,colorlinks=true,citecolor=blue,linkcolor=blue]{hyperref}

\usepackage{hyperref}
\usepackage{cleveref}
\usepackage{mathtools}
\usepackage{relsize}
\usepackage{amsfonts}
\usepackage{bm}
\usepackage{changepage}
\usepackage[usenames,dvipsnames,svgnames,table]{xcolor}
\usepackage{thmtools}
\usepackage{thm-restate}
\usepackage{subcaption}
\usepackage{enumitem}
\usepackage{algorithm}
\usepackage{algpseudocode}
\usepackage{xspace}
\AtEveryBibitem{%
  \clearfield{note}%
  \clearlist{language}
}

\usepackage{subcaption}
\usepackage{nicematrix}
\usepackage{arydshln}

\usepackage{tikz}
\usetikzlibrary{positioning}
\usetikzlibrary{calc}
\usetikzlibrary{decorations.shapes}
\usetikzlibrary{decorations.pathreplacing,angles,quotes}
\tikzset{vertex/.style={circle,fill=black,inner sep=1.6pt,draw}, edge/.style={-}}

\usepackage{todonotes}

\allowdisplaybreaks

\definecolor{blueviolet}{RGB}{60,50,200}
\definecolor{oliveg}{RGB}{40,200,30}
\hypersetup{colorlinks=true,       
	linkcolor=blueviolet,
	citecolor=oliveg,
}

\newtheorem{theorem}{Theorem}[section]
\theoremstyle{definition}
\newtheorem{definition}[theorem]{Definition}

\newtheorem{corollary}[theorem]{Corollary}
\newtheorem{claim}[theorem]{Claim}

\newtheorem{remark}[theorem]{Remark}
\newtheorem{example}[theorem]{Example}



\algrenewcommand\alglinenumber[1]{\footnotesize #1.}

\newcommand{\comment}[1]{}

\renewcommand{\E}{\mathbb{E}}

\newcommand{\Erdos}{Erd\H{o}s\xspace}
\newcommand{\Renyi}{R\'enyi\xspace}
\newcommand{\ER}{\Erdos{}-\Renyi{}}

\newcommand{\fp}{\mathrm {fp}}

\renewcommand\thefootnote{\textcolor{red}{\arabic{footnote}}}
\makeatletter
\newcommand\footnoteref[1]{\protected@xdef\@thefnmark{\ref{#1}}\@footnotemark}
\makeatother

\definecolor{DSgray}{cmyk}{0,0,0,0.7}

\usepackage{fullpage}
\usepackage{enumitem}
\usepackage{bbm}
\usepackage{amsfonts}
\usepackage{xcolor}
\usepackage{amsfonts,amsthm,amsmath,amssymb}
\usepackage{cleveref}
\usepackage[table]{xcolor}
\newcolumntype{s}{>{\columncolor[HTML]{AAACED}} m{1cm}}
\arrayrulecolor[HTML]{DB5800}

\let\geq=\geqslant %
\let\le=\leqslant 
\let\ge=\geqslant %

\renewcommand{\Pr}{\mathbb P}
\renewcommand{\thefootnote}{\fnsymbol{footnote}}

\date{}
\begin{document}
    \title{%
        The Mixed Birth-death/death-Birth Moran Process
    }
    \author{
        David A. Brewster\thanks{%
            John A. Paulson School of Engineering and Applied Sciences, Harvard University, Boston, MA 02134, USA
        }\,
        \thanks{
            Department of Molecular and Cellular Biology, Harvard University, Cambridge, MA 02138, USA
        }\,
        \thanks{%
            Department of Organismic and Evolutionary Biology, Harvard University, Cambridge, MA 02138, USA
        }
        \quad
        Yichen Huang\footnotemark[1]
        \quad
        Michael Mitzenmacher\footnotemark[1]
        \quad
        Martin A. Nowak\footnotemark[3]\,
        \thanks{
            Department of Mathematics, Harvard University, Cambridge, MA 02134, USA
        }
    }
    \maketitle
\thispagestyle{empty}

\begin{abstract}
We study evolutionary dynamics on graphs in which each step consists of one birth and one death, referred to generally as Moran processes. In standard simplified models, there are two types of individuals: residents, who have a fitness of $1$, and mutants, who have a fitness of $r$.  Two standard update rules are used in the literature. In \emph{Birth–death (Bd)}, a vertex is chosen to reproduce proportional to fitness, and one of its neighbors is selected uniformly at random to die and be replaced by the offspring. In \emph{death–Birth (dB)}, a vertex is chosen uniformly to die, and then one of its neighbors is chosen, proportional to fitness, to place an offspring into the vacancy. 
Two crucial quantities are: the unconditional absorption time, which is the expected time until only residents or only mutants remain, and the fixation probability of the mutant, which is the probability that at some time the mutants occupy the whole graph. Birth-death and death-Birth rules can yield significantly different outcomes regarding these quantities on the same graph, rendering conclusions dependent on the update rule.

We formalize and study a unified model, the $\lambda$-mixed Moran process, in which each step is independently a Bd step with probability $\lambda\in[0,1]$ and a dB step otherwise. We analyze this mixed process and establish a few results that form a starting point for its further study. 
All of our results are for undirected, connected graphs. As an interesting special case, we show at $\lambda=1/2$ for any graph that the fixation probability when $r=1$ with a single mutant initially on the graph is exactly $1/n$, and also at $\lambda=1/2$ that
the absorption time for any $r$ is $O_r(n^4)$ (that is, with an $r$-dependent constant).  We also show results for graphs that are ``almost regular,'' in a manner defined in the paper.
We use this to show that for suitable random graphs from $G\sim G(n,p)$ and fixed $r>1$, with high probability over the choice of graph, 
the absorption time is $O_r(n^4)$, the fixation probability is $\Omega_r(n^{-2})$, and we can approximate the fixation probability in polynomial time.

Another special case is when the graph has only two possible values for the degree $\{d_1, d_2\}$ with $d_1 \le d_2$. For those graphs, we give exact formulas for fixation probabilities under $r=1$ and any $\lambda$, and establish $O_r(n^4 \alpha^4)$ absorption time regardless of $\lambda$, where $\alpha = d_2/d_1$. We also provide explicit formulas for the star and cycle under any $r$ or $\lambda$.

\end{abstract}
\pagenumbering{arabic}
\newpage
\addtocounter{page}{-1}
\renewcommand{\thefootnote}{\arabic{footnote}}

\section{Introduction}
Evolution occurs in populations of reproducing individuals. The \emph{modern synthesis} in the early $20^\text{th}$ century fused natural selection with Mendelian inheritance into a quantitative framework~\cite{huxley1942evolution}. Within this framework, \emph{population genetics} studies how alleles spread through a population, with particular focus on \emph{fixation}—the event that a mutant allele takes over a resident background~\cite{
fisher1923xxi,haldane1927mathematical,Wright_1931,moran1959survival,kimura1962probability,kimura1979fixation,kimura1980average,nei1983estimation,charlesworth1997rapid,whitlock2000fixation,przeworski2003estimating,patwa2008fixation, cvijovic2015fate}. 
Evolutionary graph theory is a general framework to study evolutionary dynamics in structured populations~\cite{nowak2003linear,lieberman2005evolutionary,nowak2006evolutionary,ohtsuki_simple_2006}.
A common class of models studied in evolutionary graph theory is variations of what are referred to as \emph{Moran processes}. In this model, a population of size $n$ is represented by a (possibly weighted, directed) graph on $n$ locations, where edges constrain who can replace whom\footnote{The original Moran process, described by Patrick Moran in 1958~\cite{enwiki:1277295443,moran1958random}, considered a well-mixed population where each individual can replace any other individual, which is equivalent to the graph being complete. Generalizations to the process on graphs were introduced early in the development of evolutionary graph theory \cite{lieberman2005evolutionary,nowak2006evolutionary}. We follow the modern literature (e.g. \cite{diaz_approximating_2014, diaz2016absorption, goldberg2020phase}) and use the term ``Moran process” to refer to more generalized processes, where additional features such as graph structure and the order of the birth and death steps may also be considered.}~\cite{lieberman2005evolutionary}.
In this work, we consider unweighted, undirected graphs. Two types of individuals live on the graph, with one individual at each vertex: residents, who have fitness $1$, and mutants, who have fitness $r$. 

Prior works mainly consider two possible update rules that are generalizations of the original Moran process. With \emph{Birth-death\footnote{The ``B'' in ``Birth'' is uppercase while the ``d'' in ``death'' is lowercase to emphasize that selection only acts on birth, c.f. \Cref{remark: capitalise}.} (Bd)} updates, a vertex (that is, an individual) $u$ is chosen for reproduction with probability proportional to its fitness, then a neighbor $v$ of $u$ is chosen uniformly at random to die and is replaced by an offspring with the same type as $u$. In \emph{death-Birth (dB)}, a vertex $v$ is chosen uniformly to die, then a neighbor $u$ is chosen with probability proportional to its fitness to place an offspring of the same type at $v$. Two main quantities of interest are the \emph{fixation probability}, which is the probability that mutants starting at a given subset of vertices take over the whole graph before being eliminated, and the \emph{absorption time}, which is the expected number of steps before the graph only contains one type (resulting from either fixation or extinction).

Technically, both quantities can be expressed as the solutions to linear systems with $2^n$ variables by encoding all possible states. Due to the obvious intractability for solving the system directly, prior work has considered asymptotic bounds and approximation schemes (e.g., \emph{fully polynomial randomized approximation schemes, FPRAS}) for these measures \cite{ibsen2015computational,diaz_approximating_2014}. Past work in evolutionary graph theory has yielded a rich set of results related to these processes and metrics~\cite{lieberman2005evolutionary,hindersin2015most,adlam2015amplifiers,galanis2017amplifiers,pavlogiannis2017amplification,pavlogiannis2018construction,goldberg2019asymptotically,allen2020transient,sharma2022suppressors,tkadlec2019population,tkadlec2020limits,allen2017evolutionary}. However, these results are sensitive to the particular updating dynamics. The Bd and dB processes encode distinct biological stories, and it is sometimes the case that results under Bd dynamics behave very differently from results for dB dynamics~\cite{svoboda2024amplifiers}.

In this paper, we explore a smooth interpolation between the two processes, which we call the \emph{$\lambda$-mixed Moran process}. We unify these models with a single parameter $\lambda\in[0,1]$; in each step, the process performs a Bd update with probability $\lambda$ and a dB update with probability $1-\lambda$. 
While seemingly simple, the model exhibits interesting properties. For example, we give an example where the fixation probability for $\lambda \in (0,1)$ is not bounded by its values at $\lambda = 0$ and $\lambda = 1$, implying also that monotonicity in $\lambda$ does not hold.

Our main result holds for random graphs generated by the \ER{} process $G\sim G(n,p)$ for $p\ge \omega(\log n/n)$ when $r>1$. We show that, with high probability, the $\lambda$-mixed Moran process on $G$ has absorption time $O_r(n^4)$ and fixation probability $\Omega_r(n^{-2})$ for any non-empty initial set of mutants\footnote{$O_r(\cdot)$ and $\Omega_r(\cdot)$ hide factors that depend solely on $r$ and are constant when $r$ is constant.}, also implying an FPRAS for the fixation probability. This result is based on a more general analysis for ``almost regular'' graphs; we define almost regular later in the paper and show that random graphs are almost regular with high probability.
The main technical tools we use are appropriate choices of potential functions defined on the current mutant set and a drift argument that bounds the expected change of the potential function.

However, we also consider several other interesting special cases.
For example, we start by looking at the case $\lambda=1/2$, where Bd and dB steps are equally likely, to introduce our notation and methods. For this case, we establish that when $r=1$ (neutral evolution), the fixation probability starting with a single mutant at \emph{any} vertex is $1/n$, regardless of the graph structure. Furthermore, for any $r$, the expected absorption time is $O_r(n^4)$. Together, these imply an FPRAS for the fixation probability when $r\ge 1$ (again with the restriction $\lambda = 1/2$). \footnote{Note that an FPRAS for $r\ge 1$ cannot, in general, be converted into an FPRAS for arbitrary $r$ by swapping the roles of mutants and residents (so the new mutants have fitness $r'=1/r$). This transformation yields an FPRAS for the extinction probability, i.e., for $1-f$, rather than for the fixation probability $f$. When $f$ is small, a $1\pm\varepsilon$ multiplicative approximation for $1-f$ does not translate into a $1\pm\varepsilon$ multiplicative approximation for $f$.}

Another special case is when the graph has only two kinds of degrees $\{d_1, d_2\}$, which includes many graph families, such as stars, paths, and biregular graphs (bipartite graphs where vertices on the same side have the same degree). We give exact formulas for fixation probabilities under $r=1$ and any $\lambda$, and establish $O_r(n^4 \alpha^4)$ absorption time regardless of $\lambda$ and $r$, where $\alpha = d_2/d_1$. We also give explicit formulas for cycle and star graphs for any $r$ and $\lambda$.

\paragraph{Related Work}
Several aspects of the Bd and dB Moran processes have been extensively studied. It is well known that the fixation probabilities under both processes are monotone in $r$, and exact formulas for the fixation probability when $r=1$ are known \cite{lieberman2005evolutionary,ohtsuki_simple_2006} (as we describe after introducing notation in the next section). Both processes are also known to have $O_r(n^4)$ absorption time on any graph \cite{diaz_approximating_2014,durocher2022invasion}, and therefore admit an FPRAS for the fixation probability when $r\ge 1$. For the Bd update, \cite{goldberg2020phase} gives tighter bounds for the phase transitions of the fixation probability and proves an $O(n^{3+\varepsilon})$ upper bound on the absorption time and provides a family of graphs that has $\Omega(n^3)$ absorption time, showing their analysis is nearly tight.

A few papers have studied the $\lambda$-mixed model we study here, but to the best of our knowledge, no previous work provides formal bounds on the fixation probability and absorption time for the mixed process, or examines the special cases we do. \cite{zukewich2013consolidating} studies the mixed process where fitness is not constant but derived from payoffs of games played with neighbors. They provide heuristic approximations for fixation probability and experimental results, but without theoretical guarantees. \cite{tkadlec2020limits} explores which graphs have higher fixation probability than the complete graph (so-called “amplifiers”).


It is worth noting that the dB process is the same as what is known in other communities as the voter process \cite{sood2008voter,hindersin2015most,tkadlec2020limits}. The Moran process also fits into a well-studied family of graph processes known as interacting particle systems \cite{durrett2010,durrett1993}. However, results for interacting particle systems are usually very sensitive to the update rule and hence do not transfer.

\paragraph{Organization}
\Cref{section: prelim} formalizes the mixed Moran process and describes some technical tools. \Cref{section: lambda=1/2} analyzes the unbiased mixed Moran process where $\lambda=1/2$. \Cref{section: almost-regular} examines almost-regular and random graphs. \Cref{sec: two degree} studies graphs with two values of degrees. Explicit formulas for the stars and the cycles appear in \Cref{sec: special}. \Cref{sec: conclusion} contains some concluding remarks. Omitted proof details can be found in the Appendices.

\section{Preliminaries}
\label{section: prelim}
\subsection{Models and Notation}
\paragraph{Graph notations.}
Let $G=(V,E)$ be a connected, unweighted, undirected graph representing the locations of the $n:=|V|$ individuals in a population. We use $N(u)$ to denote the neighbors of $u$, formally $N(u) = \{v\in V\mid \{u,v\}\in E\}$, and the degree of $u$ is written as $\deg_u = |N(u)|$. We use $G(n,p)$ to denote the \Erdos{}-\Renyi{} random graph with $n$ vertices, where each edge $(u,v)$ is included independently with probability $p$. For $S \subset V$ and $u \in V$, we write $S + u = S\cup \{u\}$ and $S-u = S\setminus\{u\}$.

\paragraph{The mixed Moran process.}
We consider the $\lambda$-mixed Moran process on a graph $G$. Each vertex of the graph $G$ hosts one individual, which is either a resident with fitness $1$ or a mutant with fitness $r$. We say the evolution is advantageous if $r>1$, disadvantageous if $r<1$, and neutral if $r=1$. Initially, a subset $S_0\subseteq V$ of vertices is occupied by mutants.

At discrete times $t=1,2,\ldots$, a coupled birth-and-death event occurs in one of two possible orders. The process is parameterized by $\lambda\in[0,1]$. Each step is independently a Bd update with probability $\lambda$ and a dB update with the remaining probability $1-\lambda$.
\begin{itemize}
    \item \textbf{Birth-death (Bd).}
In a \emph{Birth-death (Bd)} step, a vertex (individual) $u \in V$ is selected for reproduction with probability proportional to its fitness to reproduce. Then, a neighbor $v \in N(u)$ is chosen uniformly at random to die, and an offspring of $u$ occupies the resulting vacancy at $v$ (inheriting the type of $u$).
\item \textbf{death-Birth (dB).} In a \emph{death-Birth (dB)} step, a vertex (individual) $v\in V$ is selected uniformly at random to die. Then, a neighbor $u \in N(v)$ is selected with probability proportional to its fitness to reproduce, and a copy of the individual at $u$ replaces $v$.
\end{itemize}

\medskip

We call the set of mutants $S_t\subseteq V$ the \emph{configuration} at time $t$. Eventually, the population becomes all mutants—we call this \emph{fixation}—or all residents—we call this \emph{extinction}. In either case, we say that \emph{absorption} has occurred. We use absorption time to refer to the expected time to reach absorption. We use $\fp_G^{\lambda,r}(S)$ to denote the fixation probability when the mutants start with $S_0 = S$, and write $\fp_G^{\lambda,r}(u)$ to abbreviate $\fp_G^{\lambda,r}(\{u\})$.  We sometimes omit the parameters when they are clear from context. We may use $\fp^{Bd}$ for $\lambda = 1$, and similarly with $\fp^{dB}$ for $\lambda = 0$.

\newcommand{\moranfix}{\textsc{Mixed Moran fixation}}

We define the problem \moranfix{}.
\begin{definition}[Mixed Moran fixation]
    The problem \textsc{Mixed Moran fixation} is defined as follows: given a graph $G=(V,E)$, a parameter $\lambda$, a fitness value $r>0$, and a subset $S\subset V$, compute $\fp_G^{\lambda,r}(S)$.
\end{definition}

\begin{remark}
\label{remark: capitalise}
In both Bd and dB updates, selection using fitness acts only on the reproducing individual. Accordingly, the B (for Birth) is uppercase while the d (for death) is lowercase. It is also possible to consider whether selection acts on death, yielding up to eight distinct processes. The Bd and dB processes are the most commonly studied~\cite{diaz2013fixation,diaz_approximating_2014,nowak2006evolutionary,nowak2003linear,durocher2022invasion}, and we restrict our attention to these two rules in this paper.
\end{remark}

Before delving into specific regimes, we provide two general characteristics for the fixation probabilities in the mixed Moran process. These results were previously known for the pure Bd and dB update rules, and we generalize them to the mixed process. The proofs appear in \Cref{appendix: prelim-basic}.


\begin{restatable}[Additivity]{theorem}{Additive}
\label{thm: additive}
    For any graph $G$, when $r=1$, for all $\lambda$, \[\fp^{\lambda, 1}_G(S) = \sum_{u \in S} \fp^{\lambda,1}_G(u).\]
\end{restatable}

\begin{restatable}{theorem}{Monotonicity}
\label{thm: monotone-in-r}
    For any graph $G$, any $\lambda$, and any $S \subseteq V$, we have
    \begin{align*}
        r \ge 1 \implies \fp^{\lambda, r}_G(S) \ge \fp^{\lambda, 1}_G(S), \quad
        r \le 1 \implies \fp^{\lambda, r}_G(S) \le \fp^{\lambda, 1}_G(S).
    \end{align*}
\end{restatable}

We believe that one could prove a stronger statement that the fixation probability is monotone in $r$ and the current set $S$. While seemingly intuitive, there is actually some subtlety in proving it rigorously (cf. \Cref{remark: monotonicity-in-r} for more discussions on this). We leave proving the property for future work.

\subsection{Technical tools}
We use a potential-function (drift) method. Fix a potential function $\phi$ on subsets $S\subseteq V$. Our goal is to bound the expected one-step change of $\phi$. These theorems are relatively standard, so we defer the proofs to \Cref{appendix: prelim-technical}.

\begin{restatable}
[Absorption time, positive drift]{theorem}{Absorb}
\label{theorem: absorption time advantage}
Let $(Y_i)_{i\ge 0}$ be a Markov chain with state space $\Omega$, where $Y_0$ is chosen from some set $I\subseteq \Omega$. Suppose there exist constants $k_1,k_2>0$ and a nonnegative function $\phi:\Omega\to\mathbb{R}$ such that
\begin{itemize}[itemsep=0em]
    \item $\phi(S)=0$ for some $S\in\Omega$, and $\phi(S)=k_1$ for some $S\in\Omega$,
    \item $\phi(S)\le k_1$ for all $S\in I$, and
    \item $\mathbb{E}[\phi(Y_{i+1})-\phi(Y_i)\mid Y_i=S]\ge k_2$ for all $i\ge 0$ and all $S$ with $\phi(S)\notin\{0,k_1\}$,
\end{itemize}
then $\mathbb{E}[\tau]\le k_1/k_2$, where $\tau=\min\{i:\phi(Y_i)\in\{0,k_1\}\}$.
\end{restatable}


\begin{restatable}[Fixation probability, lower bound]{theorem}{Fix}
\label{theorem: fixation prob}
Let $(Y_i)_{i\ge 0}$ be a Markov chain with state space $\Omega$, where $Y_0$ is chosen from some set $I\subseteq \Omega$. Suppose there exist constants $k_0,k_1>0$ and a nonnegative function $\phi:\Omega\to\mathbb{R}$ such that:
\begin{itemize}[itemsep=0em]
    \item $\phi(S)=0$ for some $S\in\Omega$, and $\phi(S)=k_1$ for some $S\in\Omega$,
    \item $\phi(Y_0)\ge k_0$,
    \item $0\le \phi(S)\le k_1$ for all $S\in I$,
    \item $\mathbb{E}[\phi(Y_{i+1})-\phi(Y_i)\mid Y_i=S]\ge 0$ for all $i\ge 0$ and all $S$ with $\phi(S)\notin\{0,k_1\}$.
\end{itemize}
and let $\tau=\min\{i\ge 0:\phi(Y_i)\in\{0,k_1\}\}$,
Then
\[
\mathbb{P}(\phi(Y_\tau)=k_1)\ge \frac{k_0}{k_1},
\]
where equality holds if the second and fourth properties hold with equality, i.e., if
\[
\phi(Y_0) = k_0, \qquad \mathbb{E}[\phi(Y_{i+1})-\phi(Y_i)\mid Y_i]= 0.
\]
\end{restatable}


\medskip

\newcommand{\bdry}{\mathrm{bdry}}
\noindent
We will be analyzing the change in potential by decomposing the contribution into the drifts over edges. Let $\bdry(S) = \{(u, v) \in E\mid u \in S, v\notin S\}$; we refer to these as the {\em boundary edges} of $S$. Note that while the graph is undirected, $\bdry(S)$ contains \emph{ordered pairs $(u,v)$} so that $u$ is always a mutant and $v$ a resident. 

We will use the following edge-wise drift terms, which correspond to the contribution of the changes in $\phi$ when $(u,v) \in \bdry(S)$ are chosen as the two vertices in a step (in both orders):
\begin{align*}
    \psi_{u,v}^{Bd}(S) &= \frac{1}{w(S)}\left(\frac{r}{\deg_u}(\phi(S+v)-\phi(S))-\frac{1}{\deg_v}(\phi(S)-\phi(S-u))\right)\\
    \psi_{u,v}^{dB}(S) &= \frac{1}{n}\left(\frac{r}{w_v(S)}(\phi(S+v)-\phi(S))-\frac{1}{w_u(S)}(\phi(S)-\phi(S-u))\right)\\
    \psi_{u,v}^{\lambda}(S) &= \lambda\,\psi_{u,v}^{Bd}(S)+(1-\lambda)\,\psi_{u,v}^{dB}(S).
\end{align*}
Here $w(S)=r|S|+(n-|S|)$ is the total fitness summed over all vertices under Bd, and $w_u(S)=r\,|N(u)\cap S|+(\deg_u-|N(u)\cap S|)$ is the total fitness among neighbors of $u$ under dB. It is easy to verify that 
\[
\sum_{(u, v) \in \bdry(S_i)} \psi_{u,v}(S_i) = \E[\phi(S_{i+1})-\phi(S_i) \mid S_i].
\]

The fixation probability is a value of central interest to compute~\cite{diaz_approximating_2014, diaz2021survey, Broom2011}. However, solving for exact values of the fixation probability would usually require solving a linear system with exponentially many equations, and thus, people have been interested in designing sampling schemes for it. We note that a polynomial lower bound on fixation probability plus a polynomial absorption time implies an FPRAS for fixation probability. 

\begin{restatable}[FPRAS]{theorem}{FPRAS}
\label{thm: fpras construction}
     Suppose that a $\lambda$-mixed Moran process with fitness $r$ on a graph $G$ satisfies that, there exists constants $c_1, c_2$ such that for all initial mutant sets $S_0$,
    \begin{itemize}
        \item $\fp(S_0)$ is $\Omega(n^{-c_1})$, and
        \item the expected absorption time is $O(n^{c_2})$.
    \end{itemize}
    Then, there is an FPRAS for $\lambda$-Moran fixation on graph $G$.
\end{restatable}

\section{The case $\lambda=1/2$}
\label{section: lambda=1/2}
We first look at an interesting special case, $\lambda = 1/2$, where moves from the Bd and dB processes are equally likely. Recall that we consider graphs that are connected and undirected. When 
$\lambda=1/2$,  the simple potential function  $\phi(S)=|S|$ behaves very nicely, providing some baseline results.  

To start, we show that when $r=1$, the fixation probability is proportional to the size of the initial mutant set. In particular, the fixation probability starting with one mutant at any vertex is always $1/n$.

\begin{restatable}{theorem}{HalfFp}
\label{thm:fp-half}
At $\lambda=1/2$, when $r=1$, the fixation probability from an initial mutant set $S_0$ in a graph with $n$ vertices is exactly $|S_0|/n$.
\end{restatable}
\begin{proof}
Fix $S_0 = \{u\}$ for an arbitrary vertex $u$. Take the potential $\phi(S)=|S|$. We show that with $\lambda = 1/2$ and  $r = 1$, $\phi(S)$ is a martingale. Consider a boundary edge $(u,v) \in \bdry(S)$. We have that edge-wise drift is
\begin{align*}
    \psi_{u,v}(S)
    &= \frac{1}{2}\left[\frac{1}{n}\cdot\frac{1}{\deg_u} - \frac{1}{n}\cdot\frac{1}{\deg_v}\right]
    + \frac{1}{2}\left[\frac{1}{n}\cdot\frac{1}{\deg_v} - \frac{1}{n}\cdot\frac{1}{\deg_u}\right] 
    = 0.
\end{align*}
Thus, every boundary edge contributes zero expected drift. Summing over all boundary edges, we obtain
\[
\mathbb{E}[\phi(S_{i+1})-\phi(S_i)\mid S_i]=0.
\]
The claim follows from \Cref{theorem: fixation prob} and \ref{thm: additive}.
\end{proof}

It is known that the fixation probability in the Bd case satisfies $\fp(u) \propto \deg_u^{-1}$ ~\cite{lieberman2005evolutionary}, and in the dB case $\fp(u) \propto \deg_u$~\cite{sood2008voter}. Using this, we can readily find an example graph where the fixation probability when $\lambda = 1/2$ is strictly greater than the fixation probability for the Bd and dB models, showing it is not monotone in $\lambda$.

\begin{example}
    Consider a graph $G$  with $5$ vertices $\{0, 1, 2, 3, 4\}$ and edges \[\{(0,1), (1,2), (1,3), (1,4), (2,3), (3,4)\}.\] The degrees are $(1, 4, 2, 3, 2)$. Consider a mutant with fitness $r=1$ at (only) vertex $2$:
    \[
    \fp_G^{Bd}(2) = \frac{\frac 1 2}{1 + \frac{1}{4} + \frac{1}{2} + \frac{1}{3} + \frac{1}{2}} = \frac{6}{31}, \qquad \fp_G^{dB}(2)=\frac 2{12} = \frac 16, \qquad \fp_G^{1/2}(2) = \frac 1 5.
    \]
\end{example}

We next bound the absorption time when $\lambda=1/2$.

\begin{theorem}
\label{thm: abs-time-half}
For any graph $G$ and $r>0$, the absorption time starting from any $S_0$ is at most $O_r(n^4)$ when $\lambda=1/2$.
\end{theorem}
\begin{proof}
We consider cases based on the value of $r$. Let $\phi(S)=|S|$ and $\tau=\min\{i:\phi(S_i)\in\{0,n\}\}$. We aim to show that $\E[\tau]$ is $O_r(n^4)$.

\vspace{-0.2cm}
\paragraph{Case $r\ne 1$.} We prove the case $r>1$. The case $r<1$ is symmetric by swapping the roles of mutants and residents; after rescaling, the “mutants” have fitness $r'=1/r$, and the same argument follows.

Recall $w(S)=r|S|+(n-|S|)$ and $w_u(S)=r\,|N(u)\cap S|+(\deg_u-|N(u)\cap S|)$ for a vertex $u$. Fix any configuration $S$ with $S\neq\emptyset$ and $S\neq V$, and consider an boundary edge $(u,v) \in \bdry(S)$. Then
\begin{align*}
    \psi_{u,v}^{1/2}(S)
    &= \frac{1}{2}\left[\frac{1}{w(S)}\left(\frac{r}{\deg_u}-\frac{1}{\deg_v}\right)+\frac{1}{n}\left(\frac{r}{w_v(S)}-\frac{1}{w_u(S)}\right)\right] \\
    &= \frac{1}{2}\left[\left(\frac{1}{w(S)}\frac{r}{\deg_u}-\frac{1}{n}\frac{1}{w_u(S)}\right)+\left(\frac{1}{n}\frac{r}{w_v(S)}-\frac{1}{w(S)}\frac{1}{\deg_v}\right)\right].
\end{align*}
For all such $S$, $n-1+r\le w(S)\le (n-1)r+1$, so
\begin{align*}
\frac{1}{w(S)}\frac{r}{\deg_u}-\frac{1}{n}\frac{1}{w_u(S)}
&\ge \frac{r}{(r(n-1)+1)\deg_u}-\frac{1}{n\,\deg_u}
= \frac{1}{\deg_u}\cdot \frac{1}{n}\cdot \frac{r-1}{nr-r+1}
\ge \frac{1}{n^3}\cdot \frac{r-1}{r},\\[0.25em]
\frac{1}{n}\frac{r}{w_v(S)}-\frac{1}{w(S)}\frac{1}{\deg_v}
&\ge \frac{r}{n\,r\,\deg_v}-\frac{1}{(n-1+r)\,\deg_v}
= \frac{1}{\deg_v}\cdot \frac{1}{n}\cdot \frac{r-1}{n+r-1}
\ge \frac{1}{n^3}\cdot \frac{r-1}{r}.
\end{align*}
Therefore,
\begin{align*}
\psi^{1/2}_{u,v}(S)\ge \frac{r-1}{r n^3}.
\end{align*}
Since this holds for all boundary edges whenever $S\neq\emptyset$ and $S\neq V$, averaging over the step dynamics yields
\[
\mathbb{E}[\phi(S_{i+1})-\phi(S_i)\mid S_i]\ge \frac{r-1}{r n^3}.
\]
Because $\phi(S)=0$ at $S=\emptyset$ and $\phi(S)=n$ at $S=V$, applying \Cref{theorem: absorption time advantage} with $k_1=n$ and $k_2=(r-1)/(r n^3)$ gives
\[
\mathbb{E}[\tau]\le \frac{k_1}{k_2}=n^4\cdot \frac{r}{r-1},
\]
which is $O_r(n^4)$ for fixed $r>1$.

\vspace{-0.2cm}
\paragraph{Case $r=1$.} The $r=1$ case requires more care, since we can no longer rely on a positive expected change with $\phi(S) = |S|$. Instead, we use a quadratic potential function.

We first show that $\tau$ has finite expectation so that we can use optional stopping. Consider a block of $n$ successive states $S_k,\dots,S_{k+n-1}$ with $S_k$ nonabsorbing; then $|S_k|\ge 1$. Consider a sequence of reproduction events in which some $u\in S_k$ spreads to the whole graph. Each step happens with probability at least $n^{-2}$, so the entire sequence occurs with probability at least $p=n^{-2n}$. Thus, the expected number of blocks to reach absorption is at most $1/p<\infty$.

We now bound $\E[(\phi(S_{i+1})-\phi(S_i))^2 \mid S_i]$. For any $S \notin \{\varnothing, V\}$ and boundary edge $(u,v) \in \bdry(S)$, with probability $1/n^2$ the step happens along $(u,v)$ and $\phi$ changes by $\pm 1$, so $\E[(\phi(S_{i+1})-\phi(S_i))^2 \mid S_i] \ge 1/n^2$.

Define $g(x)=x(n-x)$ and $Z_i=g(\phi(S_i))+\frac{i}{n^2}$. For any $x$ and $\Delta\in\{-1,0,1\}$,
\[
g(x+\Delta)-g(x)=(n-2x)\Delta-\Delta^2.
\]
Therefore
\begin{align*}
\mathbb{E}[Z_{i+1}-Z_i\mid S_i]
&=\mathbb{E}[g(\phi(S_{t+1}))-g(\phi(S_i))\mid S_i]+\frac{1}{n^2}
\\&=-\mathbb{E}[g(\phi(S_{t+1}))-g(\phi(S_i))^2\mid S_i]+\frac{1}{n^2}\le 0,
\end{align*}
so $(Z_i)$ is a supermartingale with bounded one-step differences. 

Thus, applying optional stopping at the absorption time $\tau$,
\[
\mathbb{E}[Z_\tau]\le \mathbb{E}[Z_0]=g(\phi(S_0)).
\]

Since $\phi(S_\tau)\in\{0,n\}$, we have $g\!\left(\phi(S_\tau)\right)=0$ and $g(\phi(S_0))=\phi(S_0)(n-\phi(S_0)) \le n^2/4$,
\[
   \E[\tau] \le n^2  g\!\left(\phi(S_0)\right) \le \frac{n^4}{4}. \qedhere
\]
\end{proof}

Recall that the fixation probability with initial mutants $S_0$ is $|S_0|/n$ when $r=1$. Combined with \Cref{thm: monotone-in-r}, we yield the following corollary.
\begin{corollary}\label{cor:half-lower-bound}
For any graph $G$ and $r\ge 1$, we have $\fp_G^{1/2, r}(S_0)\ge |S_0|/n$.
\end{corollary}

Applying \Cref{thm: fpras construction} together with \Cref{thm: abs-time-half} and \Cref{cor:half-lower-bound} gives an FPRAS for \moranfix{} on any graph when $\lambda=1/2$ and $r \ge 1$. 


\section{Almost Regular Graphs and Random Graphs}
\label{section: almost-regular}
We now provide results for random graphs generated by the \Erdos{}-\Renyi{} model. The key property we use is that such graphs are very likely to be close to regular. 

We begin by briefly discussing regular graphs. At a high level, and related to the analysis we have used in the last section, the potential functions employed for different Moran processes are usually the (scaled) fixation probability when $r=1$. For example, the Bd process was analyzed using a potential function $\phi(u)=\deg_u^{-1}$ in~\cite{diaz_approximating_2014}, the dB process was analyzed using $\phi(u)=\deg_u$ in~\cite{durocher2022invasion}, and we used $\phi(u)=1$ in the $\lambda=1/2$ case. These potentials have the nice property that the \emph{per-edge} drift behaves nicely (e.g., the drift is positive if and only if $r>1$).


We first show that for a regular graph $G$, the fixation probability starting with an initial mutant set $S_0$ when $r=1$ is $|S_0|/n$, regardless of $\lambda$ and the position of the mutants. This could be compared with \Cref{thm:fp-half} where the same property holds for $\lambda = 1/2$ and all graphs: both cases demonstrate this property because the process balances the change in different directions, resulting in the noted invariance. Similar to \Cref{thm: abs-time-half}, we also have the absorption time is $O_r(n^4)$. Since the proofs are very similar, we defer them to \Cref{appendix-regular}.

\begin{restatable}[]{theorem}{Regular}
\label{thm: regular}
    For any regular graph $G$, any $\lambda$ and initial mutant set $S_0 \subseteq V$, we have that
 $\fp^{\lambda, r=1}_{G}(S_0) = |S_0|/n$ for $r=1$ and the absorption time is at most $O_r(n^4)$ for any $r > 0$.
\end{restatable}


It is somewhat a strong requirement to require regularity, and we may expect that if the graph is ``almost regular", things may still behave nicely.

\begin{definition}[Almost regular]
A graph is $\alpha$-almost regular if the maximum degree is at most $\alpha$ times the minimum degree.
\end{definition}

We now proceed to the main technical theorem about the mixed Moran process on almost-regular graphs.
\begin{theorem}
\label{thm: bounds alpha regular}
For an $\alpha$-almost regular graph $G$, if $r\ge \alpha^2$, then the mixed Moran process has absorption time $O_r(n^4)$ and fixation probability $\Omega_r(n^{-2})$, regardless of $\lambda\in[0,1]$.
\end{theorem}
\begin{proof}
We perform a case analysis based on whether $\lambda\ge \frac{1}{2}$.

\paragraph{Large $\lambda$ regime ($\lambda \ge 1/2$).}
Use the potential $\phi(S)=|S|$. For $u\in S$, $v\notin S$,
\begin{align*}
\psi_{u,v}^{Bd}(S)
&=\frac{1}{w(S)}\cdot\frac{r\deg_v-\deg_u}{\deg_u\deg_v},\\
\psi_{u,v}^{dB}(S)
&
=\frac{1}{n}\left(\frac{r}{w_v(S)}-\frac{1}{w_u(S)}\right)\ge \frac{\deg_u-\deg_v}{n\,\deg_u\deg_v}.
\end{align*}
Hence, for all $S\notin\{\varnothing,V\}$,
\begin{align*}
\psi_{u,v}^{\lambda}(S)
&=\lambda\,\psi_{u,v}^{Bd}(S)+(1-\lambda)\,\psi_{u,v}^{dB}(S) \\
&=(2\lambda-1)\psi_{u,v}^{Bd}(S)+(1-\lambda)(\psi_{u,v}^{Bd}(S)+\psi_{u,v}^{dB}(S)) \\
&= (2\lambda-1)\psi_{u,v}^{Bd}(S)+(2-2\lambda)\psi^{1/2}_{u, v}(S).
\end{align*}
Recall from \Cref{section: lambda=1/2} that $\psi^{1/2}_{u, v}(S) \ge (r-1)/(rn^3)$. We now bound $\psi^{Bd}_{u,v}$.

\begin{claim}
$\psi^{Bd}_{u,v}(S)\ge \frac{r-1}{r n^3}$.
\end{claim}
\begin{proof}
We may assume the graph is not regular, since otherwise the result is already implied by \Cref{thm: regular}. Thus, $\alpha\ge (n-1)/(n-2)$. For $n\ge 3$, since $r \ge \alpha^2$ and $\frac{n-1}{n-2} \ge \frac{n}{n-1}$,
\begin{align*}
n(r-\alpha)-(r-1)&\ge (n-1)\alpha^2-n\alpha+1\\&\ge (n-1)\left(\frac{n-1}{n-2}\right)^{2}-n\left(\frac{n-1}{n-2}\right)+1=\frac{n^2-3n+3}{(n-2)^{2}}\ge 0.
\end{align*}
Thus
\[
\psi_{u,v}^{Bd}(S)\ge \frac{1}{r n}\cdot\frac{r\deg_v-\deg_u}{\deg_u\deg_v}
= \frac{1}{r n}\cdot\frac{r-\deg_u/\deg_v}{\deg_u}\ge \frac{r-\alpha}{r n^2}\ge \frac{r-1}{rn^3}. \qedhere
\]
\end{proof}

Therefore,
\[
\psi_{u,v}^{\lambda}(S)\ge \left(2\lambda-1\right)\frac{r-1}{r n^3}+(2-2\lambda)\frac{r-1}{r n^3}=\frac{r-1}{r n^3}.
\]
Applying \Cref{theorem: absorption time advantage} with $k_1=n$ and $k_2=(r-1)/(r n^3)$ yields $\mathbb{E}[\tau]$ is $O_r(n^4)$.

For the fixation probability, applying \Cref{theorem: fixation prob} with $k_0 = 1$ and $k_1=n$ gives ${\fp(S_0)\ge \frac 1 n}$.

\paragraph{Small $\lambda$ regime ($\lambda \le 1/2$).}
Use the dB potential $\phi(u)=\deg_u$. Since $r\ge \alpha^2$, the Bd term can be bounded trivially
\[
\psi_{u,v}^{Bd}(S)\ge \frac{1}{w(S)}\left(\frac{r \deg_v}{\deg_u}-\frac{\deg_u}{\deg_v}\right)\ge 0.
\]
For the dB term:
\begin{itemize}
\item If $u$ (mutant) has no mutant neighbor and $v$ (resident) has no resident neighbor, then
\begin{equation}
\label{eq: bad state}
\psi_{u,v}^{dB}(S)=\frac{1}{n}\left(\frac{r \deg_v}{w_v(S)}-\frac{\deg_u}{w_u(S)}\right)
=\frac{1}{n}\left(\frac{r \deg_v}{r \deg_v}-\frac{\deg_u}{\deg_u}\right)=0.
\end{equation}

\item If $u$ has at least one mutant neighbor, then
\[
\psi_{u,v}^{dB}(S)=\frac{1}{n}\left(\frac{r \deg_v}{w_v(S)}-\frac{\deg_u}{w_u(S)}\right)
\ge \frac{1}{n}\left(\frac{r \deg_v}{r \deg_v-r+1}-\frac{\deg_u}{\deg_u}\right)\ge \frac{r-1}{r n^2}.
\]
Similarly, if $v$ has at least one resident neighbor, then
\[
\psi_{u,v}^{dB}(S)=\frac{1}{n}\left(\frac{r \deg_v}{w_v(S)}-\frac{\deg_u}{w_u(S)}\right)
\ge \frac{1}{n}\left(\frac{r \deg_v}{r \deg_v}-\frac{\deg_u}{\deg_u+r-1}\right)\ge \frac{r-1}{r n^2}.
\]
\end{itemize}

We call a configuration \emph{bad} if every mutant has only resident neighbors and every resident has only mutant neighbors, i.e., for all $u \in S$, $N(u)\subseteq V\setminus S$ and for all $v \notin S$, $N(v) \subseteq S$. In a bad state, the mutant has no selection advantage (in \Cref{eq: bad state}, we see the $r$ in the numerator and denominator cancel out). Accordingly, we have no strictly positive lower bounds on the drift. {\footnote{It may seem counterintuitive that a configuration is bad when all neighbors are of different types: in this case, the network is guaranteed to evolve instead of self-looping. However, while the state changes, we cannot ensure that it (in expectation) moves in the desired direction (i.e., increasing the potential function). Conversely, when some neighbors are of the same type, the system is more likely to remain in the same configuration, but whenever it changes, it tends to move in a favorable direction.}
}

However, we can guarantee that a bad state transits into a non-bad state in one move. Also, in a bad state $S$, $\psi_{u,v}^{\lambda}(S)\ge 0$ for all $(u,v)$ and hence
\[
\mathbb{E}[\phi(S_{i+1})-\phi(S_i)\mid S_i]\ge 0.
\]

In a non-bad state, $\psi_{u,v}^{\lambda}(S)\ge (1-\lambda)\frac{r-1}{r n^2}\ge \frac{r-1}{2 r n^2}$ for at least one edge $(u,v)$ (and they are non-negative for all $(u,v)$), so
\[
\mathbb{E}[\phi(S_{i+1})-\phi(S_i)\mid S_i]\ge \frac{r-1}{2 r n^2}.
\]

Consider a modified process with configurations $S'$ that evolves like the $\lambda$-Moran process except that, if after one step $S'$ is a non-absorbing bad configuration, it immediately evolves one more step. This modification preserves fixation probability and (up to constants) absorption time, and satisfies
\[
\mathbb{E}[\phi(S'_{i+1})-\phi(S'_i)\mid S'_i]\ge \frac{r-1}{2 r n^2}.
\]

Applying \Cref{theorem: absorption time advantage} with $k_1=n^2$, and $k_2=(r-1)/(2 r n^2)$ yields
$\mathbb{E}[\tau]$ is $O_r(n^4)$.
For the fixation probability, applying \Cref{theorem: fixation prob} with $k_0 \ge 1$ and $k_1=n^2$ gives ${\fp(S_0)\ge n^{-2}}$.
\end{proof}

Finally, we formally establish that random graphs have concentrated degrees and are hence almost regular with high probability.

\begin{theorem}[Folklore]
\label{thm:prob-of-almost-regular-graphs}
Let $G\sim G(n,p)$. Then for any constant $\alpha$,
\[
\mathbb{P}_{G\sim G(n,p)}\left(G\text{ is }\alpha\text{-almost regular}\right)\ge 1-2n\exp(-\Theta_\alpha(np)).
\]
\end{theorem}
\begin{proof}
Let the constant $c=\frac{\alpha-1}{\alpha+1}$, so $\alpha=\frac{1+c}{1-c}$. Consider $G\sim G(n,p)$. Let $\mathcal E$ be the event that for all vertices $u$ in $G$,
\[
\deg_u\in[(1-c)(n-1)p,\,(1+c)(n-1)p].
\]
Then
\begin{align*}
\mathbb{P}(\neg\mathcal E)
&\le \mathbb{P}\left(d_{\min}<(1-c)(n-1)p\right)+\mathbb{P}\left(d_{\max}>(1+c)(n-1)p\right) \\
&\le \mathbb{P}\left(\exists u:\deg_u<(1-c)(n-1)p\right)+\mathbb{P}\left(\exists u:\deg_u>(1+c)(n-1)p\right) \\
&\le n\cdot\mathbb{P}\left(\deg_u<(1-c)(n-1)p\right)+n\cdot\mathbb{P}\left(\deg_u>(1+c)(n-1)p\right) \\
&\le 2n\exp(-\Theta_c(np)),
\end{align*}
where the last step follows from Chernoff's bound \cite{Chernoff1952}.
\end{proof}

As a corollary, taking $p=1/2$ shows that almost all graphs are $\alpha$-almost regular.

\begin{corollary}
\label{thm:density-of-almost-regular-graphs}
Let $G_n$ be the set of $n$-vertex graphs and $G_n^\alpha\subseteq G_n$ the $\alpha$-almost regular ones. For any absolute constant $\alpha$,
\[
\lim_{n\to\infty}\frac{|G_n^\alpha|}{|G_n|}=1.
\]
\end{corollary}

Combining \Cref{thm: fpras construction}, \ref{thm: bounds alpha regular}, and \ref{thm:prob-of-almost-regular-graphs} yields the FPRAS for random graphs.
\begin{restatable}[Random graphs]{theorem}{RandomGraphsMain}
\label{thm:random-graphs-main}
For every fixed $r>1$, consider a random graph $G \sim G(n, p)$. Then,  with high probability ($1-\exp(\Theta_r(np/\!\log n))$), $G$ has expected absorption time $O_r(n^4)$ and fixation probability $\Omega_r(n^{-2})$, and hence admits an FPRAS for \moranfix{}, regardless of $\lambda\in[0,1]$.
\end{restatable}

\section{Bidegreed Graphs}
\label{sec: two degree}

In this section and the following section, we consider some interesting special cases of graphs. Here we consider graphs where the degrees all take on one of two possible values, also called bidegreed graphs (e.g., \cite{myrvold1987bidegreed,belardo2009bidegreed}). This class of graphs is interesting for two reasons.
First, in some sense, it is another way of examining graphs that are nearly regular, albeit in a different way than our $\alpha$-almost regular graphs.  Additionally, several natural graph types, including stars and bipartite biregular graphs, are bidegreed;  see \Cref{fig: graphs-2-deg}. 
\begin{definition}[Bidegreed graphs]
    A graph $G$ is $(d_1, d_2)$-bidegreed (or just bidegreed) for $d_1 \le d_2$ if $\deg_u \in \{d_1,d_2\}$ for all $u$. 
\end{definition}



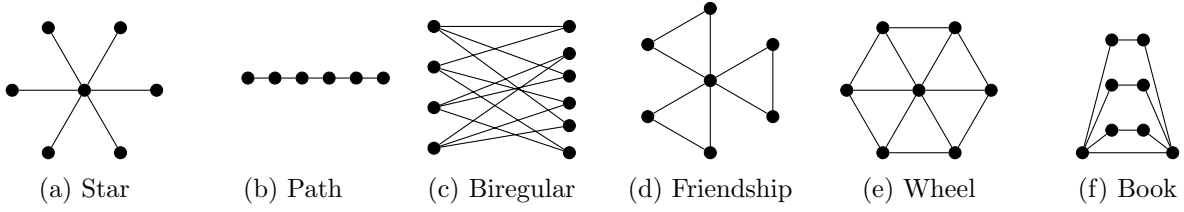
\begin{figure}[ht]
\centering

\begin{subfigure}[t]{0.16\textwidth}
\centering
\begin{tikzpicture}[scale=0.6]
  \node[vertex] (c) at (0,0) {};
  \foreach \i in {1,...,6}{
    \node[vertex] (v\i) at ({1.6*cos(60*\i)},{1.6*sin(60*\i)}) {};
    \draw[edge] (c)--(v\i);
  }
\end{tikzpicture}
\caption{Star}
\end{subfigure}\hfill%
\begin{subfigure}[t]{0.16\textwidth}
\centering
\begin{tikzpicture}[scale=0.6]
\useasboundingbox (-0.4,-0.8) rectangle (2.4,1.6);
  \foreach \i in {0,...,5}{
    \node[vertex] (p\i) at (\i*0.6,1) {};
  }
  \foreach \i in {0,...,4}{
    \draw[edge] (p\i)--(p\the\numexpr\i+1\relax);
  }
\end{tikzpicture}
\caption{Path}
\end{subfigure}\hfill%
\begin{subfigure}[t]{0.16\textwidth}
\centering
\begin{tikzpicture}[scale=0.6]
  \node[vertex] (L1) at (0,  1.4) {};
  \node[vertex] (L2) at (0,  0.5) {};
  \node[vertex] (L3) at (0, -0.4) {};
  \node[vertex] (L4) at (0, -1.3) {};
  \node[vertex] (R1) at (3,  1.4) {};
  \node[vertex] (R2) at (3,  0.8) {};
  \node[vertex] (R3) at (3,  0.3) {};
  \node[vertex] (R4) at (3, -0.3) {};
  \node[vertex] (R5) at (3, -0.8) {};
  \node[vertex] (R6) at (3, -1.4) {};
  \draw[edge] (L1)--(R1) (L1)--(R3) (L1)--(R5)
              (L2)--(R1) (L2)--(R4) (L2)--(R6)
              (L3)--(R2) (L3)--(R3) (L3)--(R6)
              (L4)--(R2) (L4)--(R4) (L4)--(R5);
\end{tikzpicture}
\caption{Biregular}
\end{subfigure}\hfill%
\begin{subfigure}[t]{0.16\textwidth}
\centering
\begin{tikzpicture}[scale=0.6]
  \node[vertex] (c) at (0,0) {};
  \foreach \k in {1,...,6}{
    \node[vertex] (w\k) at ({1.6*cos(60*\k+30)},{1.6*sin(60*\k+30)}) {};
    \draw[edge] (c)--(w\k);
  }
  \draw[edge] (w1)--(w2);
  \draw[edge] (w3)--(w4);
  \draw[edge] (w5)--(w6);
\end{tikzpicture}
\caption{Friendship}
\end{subfigure}\hfill%
\begin{subfigure}[t]{0.16\textwidth}
\centering
\begin{tikzpicture}[scale=0.6]
  \node[vertex] (c) at (0,0) {};
  \foreach \i in {1,...,6}{
    \node[vertex] (r\i) at ({1.6*cos(60*\i)},{1.6*sin(60*\i)}) {};
  }
  \foreach \i/\j in {1/2,2/3,3/4,4/5,5/6,6/1}{
    \draw[edge] (r\i)--(r\j);
  }
  \foreach \i in {1,...,6}{
    \draw[edge] (c)--(r\i);
  }
\end{tikzpicture}
\caption{Wheel}
\end{subfigure}\hfill%
\begin{subfigure}[t]{0.16\textwidth}
\centering
\begin{tikzpicture}[scale=0.6]
  \node[vertex] (u) at (-1.0,-1) {};
  \node[vertex] (v) at ( 1.0,-1) {};
  \draw[edge] (u)--(v);
  \def\sep{1}
  \def\xoff{0.35}
  \foreach \k in {-1,0,1}{
    \node[vertex] (a\k) at (-\xoff,{\k*\sep+0.5}) {};
    \node[vertex] (b\k) at ( \xoff,{\k*\sep+0.5}) {};
    \draw[edge] (u)--(a\k)--(b\k)--(v);
  }
\end{tikzpicture}
\caption{Book}
\end{subfigure}

\caption{Examples of Bidegreed Graphs.}
\label{fig: graphs-2-deg}
\end{figure}

We note that a $(d_1, d_2)$-bidegreed graph is also $(d_2/d_1)$-almost regular. Thus, the results from \Cref{section: almost-regular} apply. However, with the extra guarantee that all vertices have degrees that take on only two possible values, we achieve stronger results. We first show a closed-form expression for the fixation probability under $r=1$.

\begin{theorem}
\label{thm: two degree fp}
    Consider a $(d_1, d_2)$-bidegreed graph $G$. Then \[\fp^{\lambda, 1}_G(S_0) = \frac{\sum_{v \in S_0}f(\deg_v)}{\sum_{v\in V} f(\deg_v)},\]
    where \[f(x)= 
    \begin{cases}
        1 & x = d_1,\\
        \dfrac{\lambda d_1 + (1-\lambda)d_2}{\lambda d_2 + (1-\lambda) d_1} & x = d_2.
    \end{cases}
    \]
\end{theorem}
\begin{proof}
    We show the result for singleton sets $S = \{u\}$; the general theorem follows from \Cref{thm: additive}. We use the potential function $\phi(u) = f(\deg_u)$.
    Now fix $S$ and consider $(u, v) \in \bdry(S)$. We have 
    \[
\psi^{\lambda,1}_{u, v}(S)
=\frac{\lambda}{n}\left(\frac{f(\deg_v)}{\deg_u}-\frac{f(\deg_u)}{\deg_v}\right)
+\frac{1-\lambda}{n}\left(\frac{f(\deg_v)}{\deg_v}-\frac{f(\deg_u)}{\deg_u}\right),
\]
which easily evaluates to $0$ if $\deg_u = \deg_v$. If $(\deg_u, \deg_v) = (d_1, d_2)$, then 
\begin{align*}
    \psi^{\lambda,1}_{u, v}(S)
&=\frac{1}{n}\Bigg[\lambda\left(\frac{f(d_2)}{d_1}-\frac{1}{d_2}\right)
+(1-\lambda)\left(\frac{f(d_2)}{d_2}-\frac{1}{d_1}\right)\Bigg]
\\&= \frac{1}{nd_1d_2} (\lambda(f(d_2)d_2-d_1)+ (1-\lambda)(f(d_2)d_1-d_2))\\
&= \frac{1}{nd_1d_2}(f(d_2)(\lambda d_2+(1-\lambda)d_1)-(\lambda d_1+(1-\lambda)d_2))=0
\end{align*}
Similarly it evaluates to $0$ if $(\deg_u, \deg_v) = (d_2, d_1)$. Thus
\[\E[\phi(S_{i+1})-\phi(S_i) \mid S_i] = 0\]
and $\phi(S_i)$ is a martingale. Applying \Cref{theorem: fixation prob} with $k_0 = f(\deg_u)$ and $k_1 = \sum_v f(\deg_v)$ gives the theorem.
\end{proof}

Next, we establish a bound for the absorption time. We could apply \Cref{thm: bounds alpha regular} when $r \ge \alpha^2 = (d_2/d_1)^2$. However, for bidegreed graphs, we can remove this assumption and bound the absorption time for any $r$. We remark, however, that the exponent in our bound is probably not tight; we leave proving better bounds for future work.

\begin{theorem}
    Consider a $(d_1, d_2)$-bidegreed graph $G$, and let $\alpha = d_2/d_1$. Then for any initial set $S_0$, the absorption time is at most $O_r(n^4\alpha^2)$ if $r \ne 1$ and $O(n^4\alpha^4)$ if $r = 1$.
\end{theorem}
\begin{proof}
    We do a case analysis based on whether $r = 1$. We employ the potential $\phi(u) = f(\deg_u)$ with $f$ the same as \Cref{thm: two degree fp}.

\paragraph{Case $r \ne 1$.} We analyze the case when $r > 1$, and the $r < 1$ case again follows from swapping the role of mutants and residents. Recall that $w(S) = r|S| + (n-|S|)$ and $w_u(S) = r|N(u) \cap S| + |N(u)\setminus S|$.
We have 
\begin{align*}
    \psi^{\mathrm{Bd}}_{u, v}(S)
&=\frac{1}{w(S)}\!\left( r\,\frac{f(\deg_v)}{\deg_u}-\frac{f(\deg_u)}{\deg_v} \right)\\
&=\frac{1}{n}\!\left( \frac{f(\deg_v)}{\deg_u}-\frac{f(\deg_u)}{\deg_v} \right)
+\frac{r-1}{w(S)\,n}\!\left( (n-|S|)\frac{f(\deg_v)}{\deg_u}+|S|\frac{f(\deg_u)}{\deg_v} \right)\\
&\ge\
\underbrace{\frac{1}{n}\!\left( \frac{f(\deg_v)}{\deg_u}-\frac{f(\deg_u)}{\deg_v} \right)}_{N^{Bd}}
+\underbrace{\frac{r-1}{r\,n}\min\!\left\{\frac{f(\deg_v)}{\deg_u},\frac{f(\deg_u)}{\deg_v}\right\}}_{D^{Bd}},
\end{align*}
where we split the final expression into two terms: the first equals the potential change from the neutral $r=1$ case, and the second is an additional drift term.

For the dB case,
\begin{align*}
    \psi^{\mathrm{dB}}_{u, v}(S)
&=\frac{1}{n}\!\left( r\,\frac{f(\deg_v)}{w_v(S)}-\frac{f(\deg_u)}{w_u(S)} \right)
\\&=\frac{1}{n}\!\left( \frac{f(\deg_v)}{\deg_v}-\frac{f(\deg_u)}{\deg_u} \right)
+\frac{r-1}{n}\!\left(
\frac{f(\deg_v)|N(v)\setminus S|}{\deg_v\,w_v(S)}
+\frac{f(\deg_u)|N(u) \cap S|}{\deg_u\,w_u(S)}
\right)\\
&\ge
\underbrace{\frac{1}{n}\!\left( \frac{f(\deg_v)}{\deg_v}-\frac{f(\deg_u)}{\deg_u} \right)}_{N^{dB}}
+\underbrace{\frac{r-1}{rn^2}\!\left(\frac{f(\deg_v)|N(v) \setminus S|}{\deg_v} +\frac{f(\deg_u)\,|N(u) \cap S|}{\deg_u}\right)}_{D^{dB}}
\end{align*}

We now compute \[\psi^{\lambda}_{u, v}(S) = \lambda\psi^{Bd}_{u,v}(S) + (1-\lambda)\psi^{dB}_{u, v}(S) = \lambda N^{Bd} + (1-\lambda)N^{dB} + \lambda D^{Bd} + (1-\lambda)D^{dB}.\]

Recall from \Cref{thm: two degree fp} that the neutral terms cancel each other:
\[
\lambda N^{Bd} + (1-\lambda)N^{dB} = \frac{1}{n} \left[ \lambda \!\left( \frac{f(\deg_v)}{\deg_u}-\frac{f(\deg_u)}{\deg_v} \right) + (1-\lambda) \left(\frac{f(\deg_v)}{\deg_v}-\frac{f(\deg_u)}{\deg_u} \right)\right] = 0.
\]

It remains to provide a strictly positive lower bound for the drift terms.

Recall from the proof of \Cref{thm: bounds alpha regular} that we say a configuration $S$ is bad if $N(u) \subseteq V\setminus S$ for all $u \in S$ and $N(v) \subseteq S$ for all $v \notin S$. In that case, $|N(v) \setminus S| = |N(u) \cap S| = 0$ and we have no strictly positive lower bound for dB. Consequently, we only have $\psi^{\lambda}_{u, v}(S) \ge 0$, but we can use that $S$ will transit into a non-bad configuration in one step.

Suppose $S$ is not bad, then there exists $(u, v) \in \bdry(S)$ such that $|N(v) \setminus S| + |N(u)\cap S| \ge 1$, and so
\[
D^{dB} \ge \frac{r-1}{rn^2} \min \left\{\frac{f(\deg_v)}{\deg_v}, \frac{f(\deg_v)}{\deg_u}\right\}.
\]

Thus, 
noting that $f(\cdot) \in [\alpha^{-1}, \alpha]$,
\begin{align*}
    \psi^{\lambda}_{u, v}(S) &= \lambda N^{Bd} + (1-\lambda)N^{dB} + \lambda D^{Bd} + (1-\lambda)D^{dB}
    \\&=\lambda D^{Bd} + (1-\lambda)D^{dB}
    \\&\ge \lambda \frac{r-1}{r\,n}\min\!\left\{\frac{f(\deg_v)}{\deg_u},\frac{f(\deg_u)}{\deg_v}\right\} + (1-\lambda) \frac{r-1}{rn^2} \min \left\{\frac{f(\deg_v)}{\deg_v}, \frac{f(\deg_v)}{\deg_u}\right\}.
    \\& \ge \frac{r-1}{rn^2} \min \left\{\frac{f(\deg_v)}{\deg_v}, \frac{f(\deg_v)}{\deg_u}, \frac{f(\deg_u)}{\deg_v}, \frac{f(\deg_u)}{\deg_u}\right\}\\
    &\ge \frac{r-1}{rn^3\alpha}.
\end{align*}

Therefore,
\begin{itemize}
    \item If $S_i$ is bad, then $\E[\phi(S_{i+1} - \phi(S_i) \mid S_i] \ge 0$ and $S_{i+1}$ is not bad.
    \item Otherwise $\E[\phi(S_{i+1} - \phi(S_i) \mid S_i] \ge \frac{r-1}{rn^3\alpha}$.
\end{itemize}

On the other hand, 
\[
\sum_{v \in V} \phi(v) \le n \frac{d_2}{d_1} = n \alpha.
\]

Using the same idea as in \Cref{thm: bounds alpha regular} of collapsing a step from a bad configuration into the next step, and applying \Cref{theorem: absorption time advantage} with $k_1 = n\alpha$ and $k_2 = \frac{r-1}{rn^3 \alpha}$, the absorption time is bounded by $O_r(n^4\alpha^2)$. 

\paragraph{Case $r=1$.}
As shown in \Cref{thm: two degree fp}, $\E\!\left[\phi(S_{i+1})-\phi(S_i)\mid S_i\right]=0$.
Let
\[
\phi(S):=\sum_{u\in S} f(\deg_u),\qquad
M:=\sum_{u\in V} f(\deg_u),\qquad
\Delta_i:=\phi(S_{i+1})-\phi(S_i),
\]
and define $g(x)=x\,(M-x)$. We have $g(x) \le M^2/4$. Then
\begin{align*}
g\!\left(\phi(S_{i+1})\right)-g\!\left(\phi(S_i)\right)
&=\left(M-2\,\phi(S_i)\right)\Delta_i-\Delta_i^{\,2}.
\end{align*}
Using $\E[\Delta_i\mid S_i]=0$,
\begin{align*}
\E\!\left[g\!\left(\phi(S_{i+1})\right)-g\!\left(\phi(S_i)\right)\mid S_i\right]
&=-\,\E\!\left[\Delta_i^{\,2}\mid S_i\right]\le 0.
\end{align*}

In any nonabsorbing $S_i$, there exists a boundary edge $(u,v)$. The probability that they are selected (in any order) is at least $n^{-2}$, and in that case, $|\Delta_i| \ge \alpha^{-1}$. Thus
\[
\E\!\left[\Delta_i^{\,2}\mid S_i\right]\ \ge\ \frac{1}{n^2 \alpha^2}.
\]

Setting
\[
Z_i\ :=\ g\!\left(\phi(S_i)\right)\ + \frac i {n^2 \alpha^2},
\]
we have
\[
\E\!\left[Z_{i+1}-Z_i\mid S_i\right]\le 0,
\]
so $(Z_i)$ is a supermartingale with bounded one–step differences. Similar to the proof of \Cref{thm: abs-time-half}, it is easy to verify that $\tau$ has finite expectation. Thus, applying optional stopping at the absorption time $\tau$,
\begin{align*}
\E[g(\phi(S_\tau)] +\frac{\E[\tau]}{n^2 \alpha^2} = \E[Z_\tau] \le \E[Z_0].
\end{align*}

Since $\phi(S_\tau)\in\{0,M\}$, we have $g\!\left(\phi(S_\tau)\right)=0$ and 
\begin{align*}
   \E[\tau] \le n^2 \alpha^2 g\!\left(\phi(S_0)\right)
\le \frac{M^2n^2 \alpha^2}{4} = O(n^4 \alpha^4),
\end{align*}
since $M \le n\alpha$.
\end{proof}

As in previous sections, this implies FPRAS for \moranfix{} on bidegreed graphs when $r \ge 1$.

\section{Exact Formulations for Stars and Cycles}
\label{sec: special}
Stars and cycles, besides being basic graph types worth examining, are interesting special cases because, based on their symmetry, the state of the mutant configuration set $S$ has a simple representation. 
When a cycle graph starts with a single mutant, the mutants are always contiguous on the cycle, so the state can be represented by the number of mutants.  The star has similar symmetry, except for the center; the state can be represented by the number of mutants and whether the center is a mutant.  We use this to determine exact formulations for the fixation probability.

\subsection{Cycles}
Note that cycles are $2$-regular graphs, thus all results from \Cref{section: almost-regular} apply. On top of that, we can solve for the exact fixation probability for any $r$ and the graph starts with one mutant vertex. In this setting, we only need to control the current number of mutants, which are contiguous on the cycle. We compute the probability of the mutant size increasing $p_k^{\uparrow}$ or decreasing $p_k^{\downarrow}$ when there are $k$ mutants.

\begin{figure}[t]
    \centering
    \includegraphics[width=0.6\linewidth]{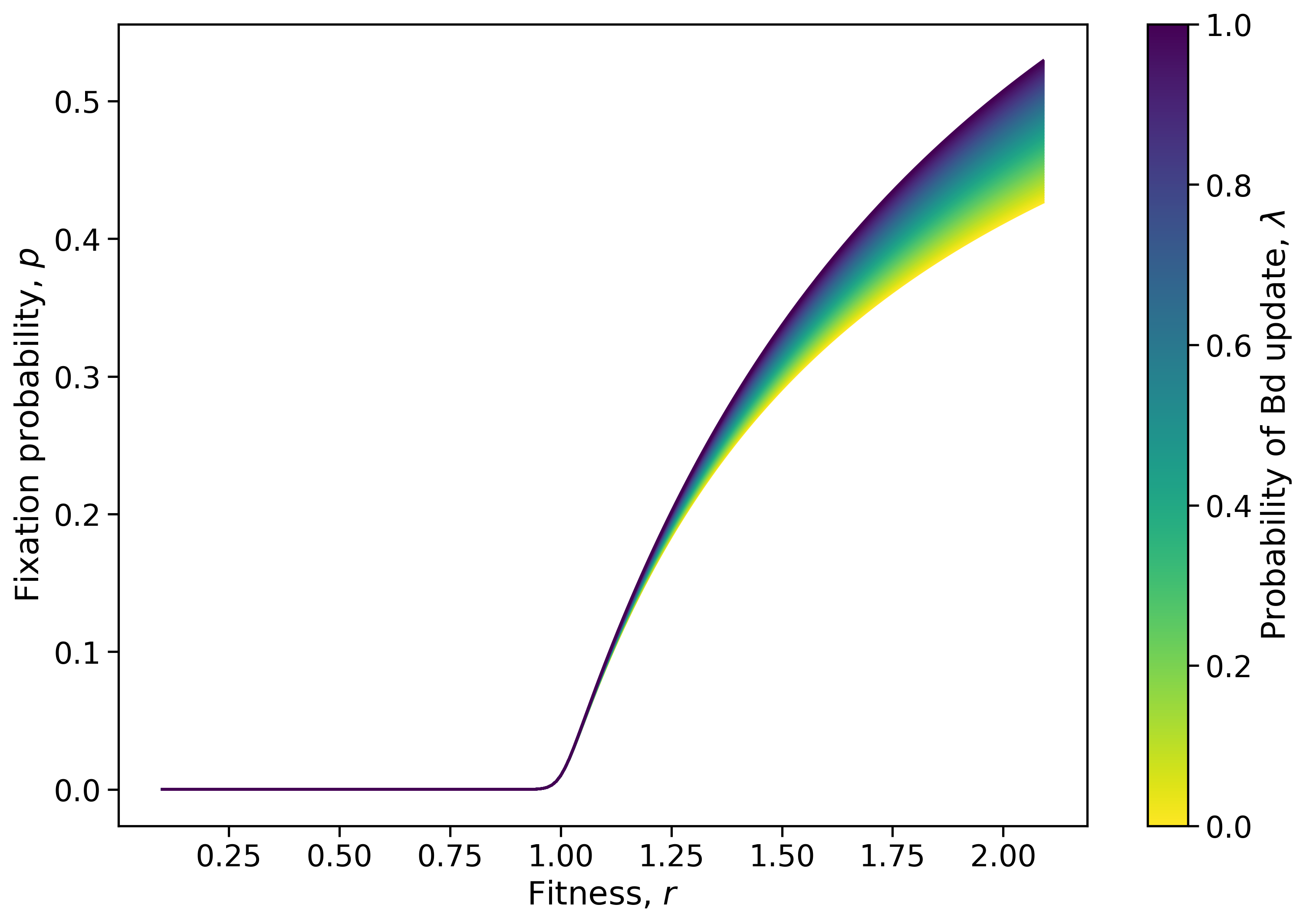}
    \caption{Fixation probabilities for a cycle on $n=100$ vertices starting from a single mutant.
    Fixation probabilities are computed using \Cref{eq: cycle value}. The two parameters are $r$ and $\lambda$, with $\lambda$ encoded by color. When $r<1$, the fixation probabilities are very small (but not actually 0).}
    \label{fig:cycle-N100}
\end{figure}

Let $F_k := rk + (n-k)$ be the total fitness of the population. We have 
\begin{itemize}
    \item For $p_k^{\uparrow}$, when $k < n-1$,
    \[
    p_k^{\uparrow} 
    = \lambda \cdot \frac{2r}{F_k} \cdot \frac{1}{2} 
      + (1-\lambda) \cdot \frac{2}{n} \cdot \frac{r}{1+r} 
    = \lambda \cdot \frac{r}{F_k} 
      + (1-\lambda) \cdot \frac{2r}{(1+r)n}.
    \]
    When $k=n-1$,
    \[
    p_k^{\uparrow} 
    = \lambda \cdot \frac{2r}{F_k} \cdot \frac{1}{2} 
      + (1-\lambda) \cdot \frac{1}{n} 
    = \lambda \cdot \frac{r}{F_k} 
      + (1-\lambda)\cdot \frac{1}{n}.
    \]
    \item For $p_k^{\downarrow}$, when $k > 1$,
    \[
    p_k^{\downarrow} 
    = \lambda \cdot \frac{2}{F_k} \cdot \frac{1}{2} 
      + (1-\lambda) \cdot \frac{2}{n} \cdot \frac{1}{1+r}
    = \lambda \cdot \frac{1}{F_k} 
      + (1-\lambda) \cdot \frac{2}{(1+r)n}.
    \]
    When $k=1$,
    \[
    p_k^{\downarrow} 
    = \lambda \cdot \frac{2}{F_k} \cdot \frac{1}{2} 
      + (1-\lambda) \cdot \frac{1}{n}
    = \lambda \cdot \frac{1}{F_k} 
      + (1-\lambda) \cdot \frac{1}{n}.
    \]
\end{itemize}
\noindent
From this, we can use formula (6.10) in \cite{nowak2006evolutionary} to get the fixation probability
\begin{equation}
\label{eq: cycle value}
    \fp^{\lambda, r}_G(u) = \frac{1}{1 + \sum_{j=1}^{n-1}\prod_{k=1}^j \gamma_k},
\end{equation}
where
\begin{align*}
    \gamma_k = p_k^{\downarrow}/p_k^{\uparrow}=
    \begin{cases}
        \dfrac{\lambda/F_k + (1-\lambda)/n}
              {\lambda r/F_k + (1-\lambda)\, 2r/((1+r)n)} 
        & \text{if } k=1, \\[1.2em]
        \dfrac{\lambda/F_k + (1-\lambda)\, 2/((1+r)n)}
              {\lambda r/F_k + (1-\lambda)/n} 
        & \text{if } k=n-1, \\[1.2em]
        \dfrac{\lambda/F_k + (1-\lambda)\, 2/((1+r)n)}
              {\lambda r/F_k + (1-\lambda)\, 2r/((1+r)n)} 
        & \text{otherwise}.
    \end{cases}
\end{align*}

\subsection{Stars}
\label{section: star}
\newcommand{\Ps}[1]{P_{#1}^{\star}}
\newcommand{\Pe}[1]{P_{#1}^{\varnothing}}
\newcommand{\Fstar}[1]{F_{#1}^{\star}}
\newcommand{\Femp}[1]{F_{#1}^{\varnothing}}
\newcommand{\Gi}{G_i}

We consider a star graph with a center vertex and $N$ leaves, so $n=N+1$ and $|E|=N$. A state in our Markov chain is described by $(i,\sigma)$, where $i\in\{0,\dots,N\}$ counts mutant leaves and $\sigma\in\{\star,\varnothing\}$ records whether the center is a mutant ($\star$) or a resident ($\varnothing$). We write
\begin{equation}
\Ps{i}:=\fp(i,\star),
\qquad
\Pe{i}:=\fp(i,\varnothing).
\end{equation}
Boundary conditions from mutant extinction/fixation are
\begin{equation}
\Pe{0}=0,
\qquad
\Ps{N}=1.
\end{equation}

Let $d_i:=N-i$ be the number of resident leaves. The fitness totals we use are
\[
\Fstar{i}=ri+d_i+r,
\qquad
\Femp{i}=ri+d_i+1,
\qquad
\Gi=ri+d_i,
\]
where $\Fstar{i}$ is the population fitness when the center is mutant, $\Femp{i}$ when the center is resident, and $\Gi$ is the leaf-layer fitness after the center dies in a dB step.

We follow a similar analysis as Section 5 in \cite{broom2008analysis}. When the center is a mutant, enumerating the one-step moves yields
\[
(1-C_i)\,\Ps{i}=A_i\,\Ps{i+1}+B_i\,\Pe{i},
\]
where
\begin{align*}
    A_i&=\lambda\,\frac{r\,d_i}{N\,\Fstar{i}}+(1-\lambda)\,\frac{d_i}{n},\\
B_i&=\lambda\,\frac{d_i}{\Fstar{i}}+(1-\lambda)\,\frac{d_i}{n\,\Gi},
\\
C_i&=\lambda\Bigl(\frac{ri}{\Fstar{i}}+\frac{ri}{N\,\Fstar{i}}\Bigr)
     +(1-\lambda)\Bigl(\frac{i}{n}+\frac{ri}{n\,\Gi}\Bigr).
\end{align*}

\begin{figure}
     \centering
     \begin{subfigure}[b]{0.45\textwidth}
         \centering
         \includegraphics[width=\textwidth]{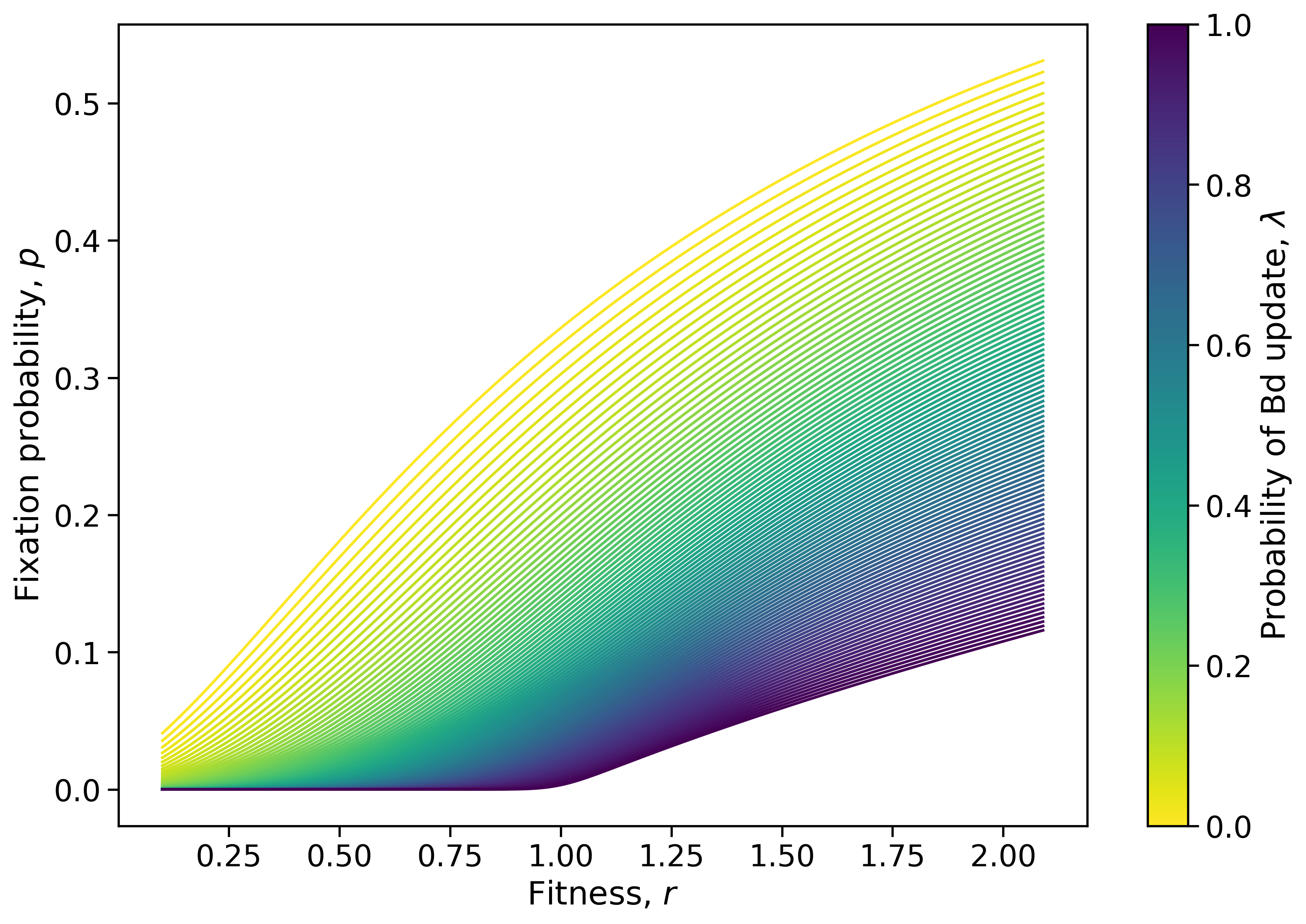}
         \caption{Initial mutant in the center}
     \end{subfigure}
     \hfill
     \begin{subfigure}[b]{0.45\textwidth}
         \centering
         \includegraphics[width=\textwidth]{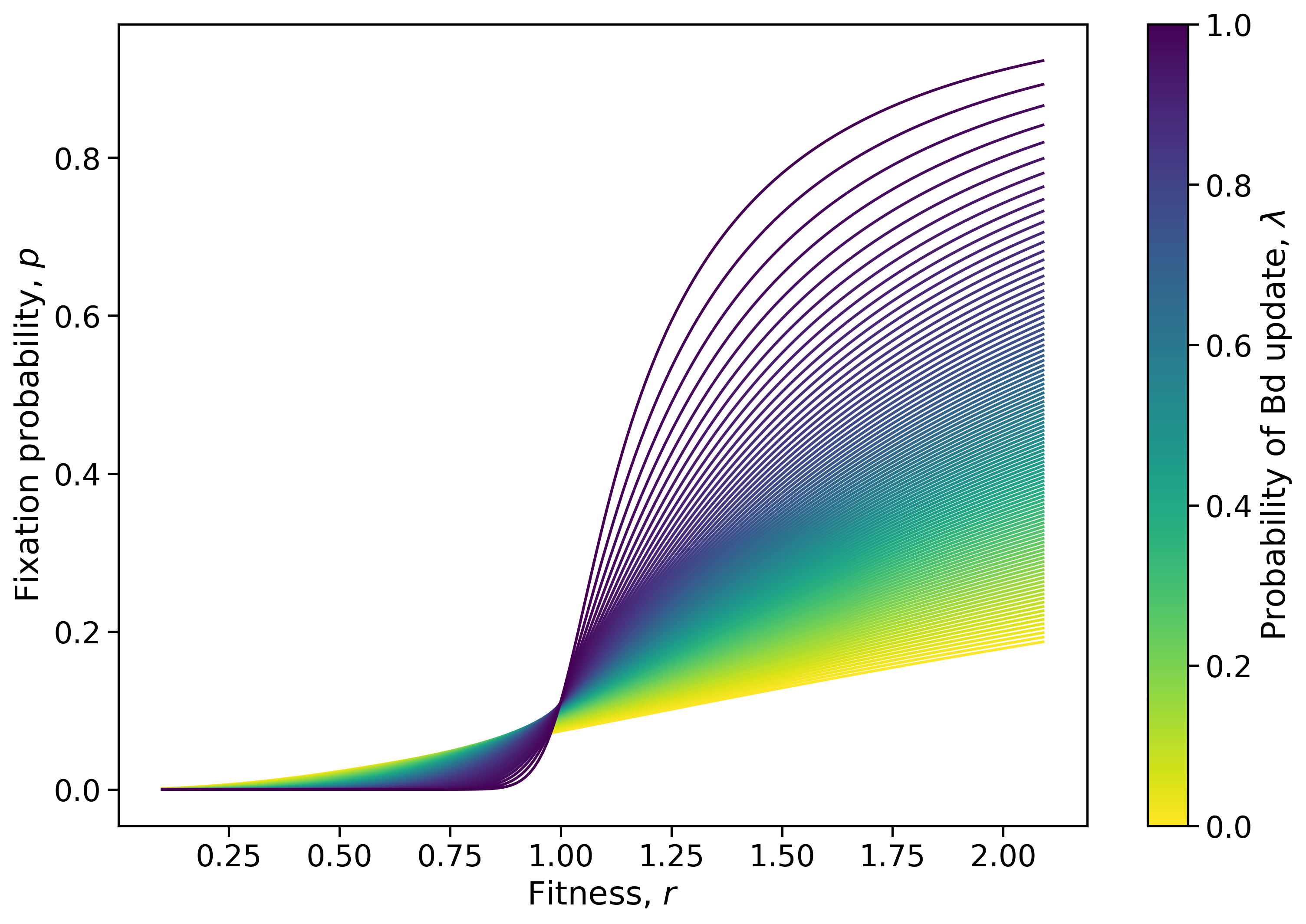}
         \caption{Initial mutant on a leaf}
         \label{fig:star-pe1}
     \end{subfigure}
    \caption{Fixation probabilities for a star on $10$ vertices starting from a single initial mutant, computed using the results from \Cref{section: star}.
    The two parameters are $r$ and $\lambda$, with $\lambda$ encoded by color. When the initial mutant is at the center, larger $\lambda$ (more Bd) consistently yields a smaller fixation probability. When starting from a leaf, there appears to be a phase transition around $r=1$ about whether greater Bd or dB gives a higher fixation probability.}
\end{figure}

Solving for $\Ps{i}$ isolates the unknowns on the right:
\[
\Ps{i}=\alpha_i\,\Ps{i+1}+\beta_i\,\Pe{i},
\qquad
\alpha_i=\frac{A_i}{1-C_i},
\qquad
\beta_i=\frac{B_i}{1-C_i}.
\]

When the center is resident, we similarly obtain
\[
(1-b_i)\,\Pe{i}=a_i\,\Ps{i}+c_i\,\Pe{i-1},
\]
where
\[
a_i=\lambda\,\frac{ri}{\Femp{i}}+(1-\lambda)\,\frac{ri}{n\,\Gi},
\qquad
c_i=\lambda\,\frac{i}{N\,\Femp{i}}+(1-\lambda)\,\frac{i}{n},
\qquad
b_i=1-a_i-c_i,
\]
and thus
\[
\Pe{i}=p_i\,\Ps{i}+q_i\,\Pe{i-1},
\qquad
p_i=\frac{a_i}{1-b_i},
\qquad
q_i=\frac{c_i}{1-b_i}.
\]

Combining the two recurrences gives
\[
\Ps{i+1}
= \frac{1-\beta_i p_i}{\alpha_i}\,\Ps{i} - \frac{\beta_i q_i}{\alpha_i}\,\Pe{i-1}.
\]
Equivalently,
\[
\begin{pmatrix}
\Ps{i+1}\\[2pt] \Pe{i}
\end{pmatrix}
= M_i
\begin{pmatrix}
\Ps{i}\\[2pt] \Pe{i-1}
\end{pmatrix},
\qquad
M_i=\begin{pmatrix}
\frac{1-\beta_i p_i}{\alpha_i} & -\frac{\beta_i q_i}{\alpha_i}\\[4pt]
p_i & q_i
\end{pmatrix}.
\]
Iterating,
\[
\begin{pmatrix}
\Ps{N}\\[2pt] \Pe{N-1}
\end{pmatrix}
=\left(\prod_{j=1}^{N-1} M_j\right)
\begin{pmatrix}
\Ps{1}\\[2pt] \Pe{0}
\end{pmatrix},
\qquad
A^{(i)}:=\prod_{j=1}^{i} M_j.
\]
Since $\Ps{N}=1$ and $\Pe{0}=0$,
\[
\Ps{1}=\frac{1}{A^{(N-1)}_{1,1}},
\qquad
\Ps{0}=\alpha_0\,\Ps{1},
\qquad
\Pe{1}=p_1\,\Ps{1}.
\]

\section{Conclusion and Discussion}
\label{sec: conclusion}
In this paper, we formalize and study the $\lambda$-mixed Moran process, a simple model that interpolates between Birth–death and death–Birth dynamics. We provide several results for the fixation probability and absorption time, in particular for random graphs and various special cases. Many interesting problems remain open; we outline a few below.
\begin{enumerate}
    \item Our bounds on absorption times and fixation probabilities are proved for special regimes. Can one establish analogous bounds for general graphs, as in the pure Bd and dB cases? For the Bd update, a highly involved analysis shows upper and lower bounds of $O(n^{3+\epsilon})$ and $\Omega(n^3)$, respectively~\cite{goldberg2020phase}, which are nearly tight. Could similar tight characterizations be achieved for the mixed Moran process? Note that even for the pure dB process, near-tight bounds are not known.
    \item We restricted our attention to unweighted, undirected graphs. What happens for directed graphs and weighted graphs? On directed weighted graphs, the Bd process, after selecting the birth vertex $u$ with probability proportional to the fitness, selects an out edge proportional to the edge weight; the dB process selects a death vertex $v$ uniformly at random and then selects among in-edges $(u, v)$ with probability proportional to the fitness of $u$ times the weight of the edge~\cite{lieberman2005evolutionary}. This introduces additional asymmetry between Bd and dB; how will the mixed process behave?
    \item In the Bd process, it is known that for a class of graphs called isothermal graphs (where a \emph{death} occurs at every vertex with equal probability when $r = 1$)~\cite{lieberman2005evolutionary}, each vertex has the same fixation probability as in a complete graph for all $r$. Is there an analogous class of graphs for the mixed process?
    \item Another question related to the Moran process is the design of suppressors and amplifiers. For $r>1$, a graph is a suppressor (resp., amplifier) if the fixation probability of a uniformly placed mutant in the graph is lower (resp., higher) than that in a complete graph~\cite{tkadlec2021fast,adlam2015amplifiers,galanis2017amplifiers,goldberg2020phase}. It would be interesting to study these structures for the mixed process, especially given that amplifiers and suppressors sometimes look different between the Bd and dB processes~\cite{svoboda2024amplifiers}.
\end{enumerate}

Overall, we believe that the $\lambda$-mixed process offers a principled bridge between the Bd and dB models, and we hope this perspective stimulates further study of evolutionary dynamics under mixed updating rules.

\section*{Acknowledgements}
The authors thank the anonymous reviewers of ITCS 2026 for their useful comments. YH and MM are supported by NSF grant CNS-2107078. DAB is supported by a Harvard Graduate School of Arts and Sciences Prize Fellowship.


\printbibliography

\appendix

\section{Proofs for Preliminaries}
\subsection{Basic Properties of mixed Moran Process}
\label{appendix: prelim-basic}

\Additive*
\begin{proof}
    Consider a modified process, where instead of two types, there exist $n$ types, all of which have a fitness of $1$. Initially, vertex $i$ starts with type $i$. The process evolves as the $\lambda$-mixed process with fitness $1$:
    \begin{itemize}
        \item A Bd step happens with probability $\lambda$, selecting a vertex to birth uniformly at random and replacing a random neighbor.
        \item A dB step happens with probability $1-\lambda$, selecting a vertex to die uniformly at random and adopting a random neighbor.
    \end{itemize}
    The process terminates when the whole graph contains a single type. Recall $\fp(S)$ is the fixation probability in the original process for a mutant currently occupying subset $S$, and we denote by $\widetilde {\fp}(u)$ the probability that type $u$ takes over the whole graph in the modified process. For any vertex $u$, by viewing all types $v \neq u$ as residents and $u$ as the mutant, we have $\widetilde{\fp}(u) = \fp(u)$. By a coupling argument, the probability that the probability of a mutant currently occupying $S$ fixating equals the probability that some $u \in S$ takes over, i.e., 
    \[\fp(S) = \sum_{u \in S} \widetilde{\fp}(u) = \sum_{u \in S} \fp(u). \qedhere\] 
\end{proof}
\Monotonicity*
\begin{proof}
We prove the case for $r \ge 1$, and the case when $r \le 1$ is implied by swapping the role of mutants and residents and setting fitness of the mutant as $1/r$.

We use a coupling argument. Let $(Z_t)_{t \ge 0}$ denote the neutral process with fitness $1$ and $(X_t)_{t \ge 0}$ denote the advantageous process with fitness $r > 1$, both starting from the same set $Z_0 = X_0 = S$. We construct a coupling such that $Z_t \subseteq X_t$ for all steps $t$. Consider the transition at time $t$. For $u \in V, v \in N(u)$, we denote by $q_t^Z(u \to v)$ the probability that vertex $u$ is selected for birth and replaces neighbour $v$ in $Z_t$, and $q_t^{Z}(u \leftarrow v)$ the probability that vertex $u$ is selected for death and is replaced by neighbour $v$ in $Z_t$. Similarly, $q_t^X$ denotes the same quantity for the chain $(X_t)$.
\begin{itemize}
    \item If $Z_t$ performs a Bd update with a mutant $u$ replacing a neighbour $v$, $X_t$ also performs a Bd update with mutant $u$ replacing $v$. This is possible since
    \[
    q_t^Z(u \to v) = \frac{\lambda}{n \deg_u} \le \frac{\lambda}{\deg_u} \frac{r}{w(X_t)} = q_t^X(u \to v).
    \]
    \item If $X_t$ performs a Bd update with resident $v$ replacing a neighbour $u$, $Z_t$ also performs a Bd update with resident $v$ replacing $u$. This is possible since
    \[
    q_t^X(v \to u) = \frac{\lambda}{\deg_v}\frac{1}{w(X_t)} \le \frac{\lambda}{n \deg_v} = q_t^Z(v \to u).
    \]
    \item If $Z_t$ performs a dB update with a vertex $v$ replaced by a mutant neighbour $u$, $X_t$ also does that. This is possible since 
    \[
    q_t^Z(v \leftarrow u) = \frac{1-\lambda}{n \deg_v} \le \frac{1-\lambda}{n} \frac{r}{w_v(X_t)} = q_t^X(v \leftarrow u).
    \]
    \item If $X_t$ performs a dB update with a vertex $u$ replaced by a resident neighrbour $v$, $Z_t$ also does that. This is possible since
    \[
    q_t^X(u \leftarrow v) =\frac{1-\lambda}{n} \frac{1}{w_v(X_t)}  \le \frac{1-\lambda}{n \deg_v} = q_t^Z(u \leftarrow v).
    \]
    \item All remaining probabilities place a mutant in $X_t$ and a resident in $Z_t$, so we can couple them arbitrarily.
\end{itemize}

Following this coupling, $Z_t \subseteq X_t$ holds for all $t$. We conclude that $\fp^{\lambda, 1}_G(S) \le \fp^{\lambda, r}_G(S)$.
\end{proof}

\begin{remark}
\label{remark: monotonicity-in-r}
    We briefly explain the challenge of proving the monotonicity of fixation probability with respect to $r$ in the mixed Moran process.

    For pure Bd updates, the challenge in establishing monotonicity in $r$ is that the reproduction probability is normalized by the total fitness, and consequently, having more mutants can potentially slow down the reproduction of other mutants, thereby breaking the natural coupling (see, e.g. \cite{diaz2016absorption} for a more thorough discussion). In \cite{diaz2016absorption}, this is overcome by pivoting to a continuous-time Markov chain, where each vertex waits a random time before reproduction, and this time follows an exponential distribution parameterized by its fitness. Then, the discrete sequence of events in this continuous-time Markov chain corresponds exactly to the discrete-time Moran process. Furthermore, in this setting, the presence of additional mutants does not slow down the reproduction of others. Thus, a natural coupling demonstrates stochastic dominance in the fixation probability of the continuous embedding.

    However, for the mixed process, this solution faces another challenge: In the continuous embedding for Bd, each vertex has a rate of $f(u)$ for Birth, whereas dB has a fixed rate of $1$ for death. This is problematic since it breaks the condition that the ratio between the rates of Bd and dB events remains $\lambda/(1-\lambda)$ and hence the discretization of the continuous-time chain is not equivalent to the discrete Moran process anymore. A natural attempt is then to scale the dB rate by the total fitness to fix this ratio, but this brings back the original problem: having more mutants increases the total fitness and thus the rate of death events, which again breaks the natural coupling. Regardless, we believe that monotonicity in $r$ still holds, although establishing it likely requires a more careful analysis.
\end{remark}

\subsection{Martingales, Drifts, and Approximation Schemes}
\label{appendix: prelim-technical}
We now prove the technical lemmas. We will need the following theorem in~\cite{diaz_approximating_2014}.

\begin{theorem}[~\cite{diaz_approximating_2014}]
\label{thm: abs time dis}
Let $(Y_i)_{i\ge 0}$ be a Markov chain with state space $\Omega$, where $Y_0$ is chosen from some set $I\subseteq \Omega$. If there exist constants $k_1,k_2>0$ and a nonnegative function $\phi:\Omega\to\mathbb{R}$ such that
\begin{itemize}[itemsep=0em]
    \item $\phi(S)=0$ for some $S\in\Omega$,
    \item $\phi(S)\le k_1$ for all $S\in I$, and
    \item $\mathbb{E}[\phi(Y_i)-\phi(Y_{i+1})\mid Y_i=S]\ge k_2$ for all $i\ge 0$ and all $S$ with $\phi(S)>0$,
\end{itemize}
then $\mathbb{E}[\tau]\le k_1/k_2$, where $\tau=\min\{i:\phi(Y_i)=0\}$.
\end{theorem}

\Absorb*

\begin{proof}
    Consider the following process $Y_i'$, which behaves identical to $Y_i$, except that if $\phi(Y_i') = 0$, then $\phi(Y_{i+1}') = S$ for some $S$ such that $\phi(S) = k_1$. Let $\phi'(S) = k_1-\phi(S)$. Clearly
    \[
    \min \{i: \phi'(Y'_i) = 0\} =\min \{i: \phi(Y'_i) = k_1\} \ge \min\{i:\phi(Y_i) \in \{0, k_1\}\} = \tau.
    \]
    
    Now $\phi'(Y_i')$ satisfies the conditions in \Cref{thm: abs time dis} with the same parameters, since for all $Y_i'$ such that $\phi'(Y_i') \notin \{0, k_1\}$, the conditions holds by assumption, and if $\phi'(Y_i') = k_1$, then $\phi'(Y_{i+1}') = 0$ so the conditions also hold. Applying \Cref{thm: abs time dis} gives the claim.
\end{proof}

\Fix*
\begin{proof}
By the claimed properties, $(\phi(S_i))_{i\ge 0}$ is a submartingale with bounded value.  By the optional stopping theorem,
\[
\mathbb{P}(\phi(Y_\tau)=k_1)k_1 =
\mathbb{E}[\phi(S_\tau)] \ge \mathbb{E}[\phi(S_0)] \ge k_0.
\]
Hence \[
\mathbb{P}(\phi(Y_\tau)=k_1)\ge \frac{k_0}{k_1}.
\]

If $\phi(Y_0) = k_0$ and $\E[\phi(Y_{i+1})-\phi(Y_i)\mid Y_i]= 0$, then the above holds with equality.
\end{proof}

Finally, we show the FPRAS construction via standard techniques. The idea is to run the experiments enough times and report the empirical mean, while aborting if any run took too long.

\FPRAS*
\begin{proof}
Write $f = \fp(S_0)$ the true fixation probability for the given $\lambda$-Moran process on $G$. By the hypotheses, there exist positive constants $C_{\mathrm{fp}},C_{\tau}$ and integers $c_1,c_2$ such that
\[
f\ge C_{\mathrm{fp}}n^{-c_1}
\qquad\text{and}\qquad
\mathbb{E}[\tau]\le C_{\tau}n^{c_2},
\]
where $\tau$ is the absorption time.

We run $N$ independent simulations of the process, each stopped upon absorption or after $T$ where \[
N=\left\lceil \frac{\ln(16)}{2\varepsilon^2C_{\mathrm{fp}}^2}n^{2c_1}\right\rceil,
\qquad
T=\left\lceil 8C_{\tau}Nn^{c_2}\right\rceil.
\]
If any run hits the cutoff $T$ without absorbing, abort and return an error $\hat f = \bot$. Otherwise, output
\[
\hat f=\frac{1}{N}\sum_{i=1}^{N}X_i,
\]
where $X_i=1$ if the $i$th run fixes and $X_i=0$ otherwise.

Since each step can be simulated in $O(1)$ time, the whole algorithm runs in time $O(NT)$, which is polynomial in $n$ and $1/\varepsilon$. We now bound the accuracy.

\begin{itemize}
    \item Consider a variant of the algorithm that always simulates until absorption with output $\tilde f=\frac{1}{N}\sum_{i=1}^{N}X_i$. Then $X_1,\dots,X_N$ are i.i.d. $\mathrm{Bernoulli}(f)$. Hoeffding’s inequality implies
\[
\Pr(|\tilde f-f|>\varepsilon f)\le 2\exp(-2\varepsilon^2 f^2 N)
\le 2\exp\left(-2\varepsilon^2 C_{\mathrm{fp}}^2 n^{-2c_1} N\right)\le \tfrac{1}{8},
\]
by the choice of $N$.
\item For a single run,
\[
\Pr(\tau>T)\le \frac{\mathbb{E}[\tau]}{T}\le \frac{C_{\tau} n^{c_2}}{8 C_{\tau} N n^{c_2}}=\frac{1}{8N}.
\]
We have $\hat f = \tilde f$ if and only if no run is cut off. By a union bound, $\Pr(\hat f\ne \tilde f) \le 1/8$.
\end{itemize}

Combining the two events, the failure probability is at most
\[
\Pr(|\tilde f-f|>\varepsilon f) \le\Pr(|\tilde f-f|>\varepsilon f) + \Pr(\hat f\ne \tilde f) \le 1/4. \qedhere
\]
\end{proof}

One could amplify the success probability to $1-\delta$ by running the algorithm $O(\log (1/\delta))$ times and take the median (see, e.g. \cite{vazirani2010approx}).

\section{Regular Graphs}
\label{appendix-regular}
\Regular*
\begin{proof}
We first establish the fixation probability when $r=1$. Fix $S_0 = \{u\}$ for an arbitrary vertex $u$. Consider the potential $\phi(S)=|S|$. Since all degrees are equal,
\[
 \psi_{u,v}(S) = \lambda\left[\frac{1}{n}\cdot\frac{1}{\deg_u} - \frac{1}{n}\cdot\frac{1}{\deg_v}\right]
+ (1-\lambda)\left[\frac{1}{n}\cdot\frac{1}{\deg_v} - \frac{1}{n}\cdot\frac{1}{\deg_u}\right] 
= 0.
\]
Thus, $\phi(S)$ is a martingale. The theorem follows from \Cref{theorem: fixation prob} and \ref{thm: additive}. 

We now bound the absorption time. Let the graph $G$ be $d$-regular. Using the exact same argument as in the proof of \Cref{thm:fp-half}, we obtain that the absorption time is $O(n^3d)$ when $r=1$. It remains to show the case where $r\ne 1$. We establish the result for $r > 1$. The case $r < 1$ is symmetric by swapping the roles of mutants and residents.

Consider a boundary edge $(u,v) \in \bdry(S)$. The expected drift contributed by this edge is:
\begin{align*}
     \lambda \cdot \psi_{u, v}^{Bd}(S) + (1-\lambda) \cdot \psi_{u, v}^{dB}(S) &= \frac {\lambda} {w(S)} \left(\frac r {d} - \frac 1 {d} \right)+ \frac {1-\lambda} {n} \left(\frac r {w_v(S)} - \frac 1 {w_u(S)} \right)
\end{align*}

The Bd component is strictly positive. Since $w(S) \le nr$, we have:
\[\frac {\lambda} {w(S)} \left(\frac r {d} - \frac 1 {d} \right) \geq \frac{\lambda(r-1)}{rnd}.\]

For the dB component, we perform an analysis similar to that in the proof of \Cref{thm:random-graphs-main}. 
\begin{itemize}
    \item If all neighbors of the mutant $u$ are residents and all neighbors of the resident $v$ are mutants, then $w_u(S) = d$ and $w_v(S) = rd$. The dB term vanishes:
    \[\frac r {w_v(S)} - \frac 1 {w_u(S)} = \frac r {rd} - \frac 1 {d} = 0.\]
    \item If $u$ has at least one mutant neighbor, then $w_u(S) \ge r + (d-1)$. Thus:
    \[\frac r {w_v(S)} - \frac 1 {w_u(S)} \geq \frac r {rd} - \frac 1 {d+r-1} = \frac{1}{d} - \frac{1}{d+r-1} \geq \frac{r-1}{d(d+r)}.\]
    \item If $v$ has at least one resident neighbor, then $w_v(S) \le r(d-1) + 1$. Thus:
    \[
    \frac r {w_v(S)} - \frac 1 {w_u(S)} \geq \frac r {r(d-1)+1} - \frac 1 {d} \geq \frac{r-1}{d^2r}.
    \]
\end{itemize}

We call a configuration ``bad'' if every mutant has only resident neighbors and every resident
has only mutant neighbors. 

In a bad state, the drift is non-negative, and a bad state converts into a non-bad state in one step. Furthermore, in a non-bad state, we have
\[\E[\phi(S_{i+1})-\phi(S_i) \mid S_i] \geq \frac {r-1}{rn^2d}.\]
Consider a modified Markov chain $S'$ which ``glues'' one step in a bad configuration with one more step, which preserves the absorption time up to a constant factor and satisfies
\[\E[\phi(S_{i+1}')-\phi(S_i') \mid S_i] \geq \frac {r-1}{rn^2d}.\]

Since $\phi(V) = n$, applying \Cref{theorem: absorption time advantage} yields an absorption time of $\frac{r}{r-1} O(n^3d) = O_r(n^4)$.
\end{proof}

\end{document}